\documentclass[12pt]{amsart}

\usepackage{amsfonts, amsthm, amsmath}

\usepackage{rotating}

\usepackage{tikz}

\usepackage{graphics}

\usepackage{amssymb}

\usepackage{amscd}

\usepackage[latin2]{inputenc}

\usepackage{t1enc}

\usepackage[mathscr]{eucal}

\usepackage{indentfirst}

\usepackage{graphicx}

\usepackage{graphics}

\usepackage{pict2e}

\usepackage{mathrsfs}

\usepackage{enumerate}
\usepackage[pagebackref]{hyperref}
\hypersetup{colorlinks=true}
\usepackage{cite}
\usepackage{color}
\usepackage{epic}
\usepackage{hyperref} 
\usepackage{framed}
\usepackage{mathabx}
\usepackage{booktabs}
\usepackage{makecell}
\newcolumntype{V}{!{\vrule width 2pt}}

\numberwithin{equation}{section}
\def\blue{\textcolor{blue}}

\def\magenta{\textcolor{magenta}}

\theoremstyle{plain}

\newtheorem{theorem}{Theorem}[section]
\newtheorem{corollary}[theorem]{Corollary}
\newtheorem{proposition}[theorem]{Proposition}

\newtheorem{remark}[theorem]{Remark}
\newtheorem{lemma}[theorem]{Lemma}
\newtheorem{definition}[theorem]{Definition}

\def\P{\mathcal{P}}

\def\B{\mathcal{B}}
\def\I{\mathcal{I}}
\def\T{\mathcal{T}}

\def\S{\mathcal{S}}
\def\Z{\mathbb{Z}}
\def\A{\mathcal{A}}

\def\leaf{\mathsf{leaf}}
\def\st{\mathsf{st}}
\def\twin{\mathsf{twin}}
\def\od{\mathsf{od}}
\def\el{\mathsf{el}}
\def\deg{\mathsf{deg}}
\def\elint{\mathsf{elint}}
\def\ystleaf{\mathsf{ystleaf}}

\def\sleaf{\mathsf{sleaf}}
\def\suleaf{\mathsf{suleaf}}
\def\snuleaf{\mathsf{snuleaf}}
\def\eleaf{\mathsf{eleaf}}
\def\elleaf{\mathsf{elleaf}}
\def\yleaf{\mathsf{yleaf}}
\def\oleaf{\mathsf{oleaf}}
\def\etleaf{\mathsf{etleaf}}
\def\entleaf{\mathsf{entleaf}}
\def\etleaf{\mathsf{etleaf}}
\def\syleaf{\mathsf{syleaf}}
\def\yerleaf{\mathsf{yerleaf}}
\def\sint{\mathsf{sint}}
\def\eint{\mathsf{eint}}
\def\yint{\mathsf{yint}}
\def\oint{\mathsf{oint}}

\def\dd{\mathsf{dd}}
\def\da{\mathsf{da}}
\def\pk{\mathsf{pk}}
\def\mnd{\mathsf{mnd}}
\def\mna{\mathsf{mna}}
\def\ldr{\mathsf{ldr}}
\def\rar{\mathsf{rar}}
\def\asc{\mathsf{asc}}
\def\des{\mathsf{des}}

\def\ldr{\mathsf{ldr}}
\def\lmin{\mathsf{lmin}}
\def\rmin{\mathsf{rmin}}

\def\ASC{\operatorname{ASC}}
\def\DES{\operatorname{DES}}
\def\Unb{\operatorname{Unb}}
\def\Orl{\operatorname{Orl}}

\begin{document}

\title[Bijections in weakly increasing trees]{Bijections in weakly increasing trees via binary trees}

\author[Y. Li]{Yang Li}
\address[Yang Li]{Research Center for Mathematics and Interdisciplinary Sciences, Shandong University, \& Frontiers Science Center for Nonlinear Expectations, Ministry of Education, Qingdao 266237, P.R. China}
\email{202421349@mail.sdu.edu.cn}

\author[Z. Lin]{Zhicong Lin}
\address[Zhicong Lin]{Research Center for Mathematics and Interdisciplinary Sciences,  Shandong University \& Frontiers Science Center for Nonlinear Expectations, Ministry of Education, Qingdao 266237, P.R. China}
\email{linz@sdu.edu.cn}

\date{\today}

\begin{abstract}
As a unification of increasing trees and plane trees, the weakly
increasing trees labeled by a multiset was introduced by Lin-Ma-Ma-Zhou in 2021.
Motived by some symmetries in plane trees proved recently by Dong, Du, Ji and Zhang, we construct four bijections on weakly increasing trees  in the same flavor via switching  the role of left child and right child of some specified  nodes in their corresponding binary trees. Consequently,  bijective proofs of the aforementioned symmetries found by Dong et al. and a non-recursive construction of a bijection on plane trees of Deutsch are provided. Applications of some symmetries in weakly increasing trees to permutation patterns and statistics  will also be discussed. 
\end{abstract}

\keywords{Weakly increasing trees, Binary trees, Symmetries, Bijections, Patterns}

\maketitle

\section{Introduction}
 Using the approach of  grammatical calculus introduced by Chen~\cite{Chen1993} and further developed by Chen and Fu~\cite{Chen2017}, Dong, Du, Ji and Zhang~\cite{Dong1} proved three symmetric joint distributions over plane trees and asked for bijective proofs. Weakly increasing trees introduced by Lin, Ma, Ma and Zhou~\cite{Lin20} are common generalization of plane trees and increasing trees. The main objective of this paper is to extend the aforementioned symmetries of Dong et al. from plane trees to weakly increasing trees via  three tree involutions  constructed in a unified manner. 
 
 Recall that a {\em plane tree} is a rooted tree in which the children of each node are linearly ordered. Throughout this paper, let $M = \{1^{p_1} , 2^{p_2} , \cdots , n^{p_n} \}$ be a multiset with cardinality  $|M|= p_1 + \cdots + p_n$. 
 \begin{definition}[Weakly increasing trees~\cite{Lin20}]
     A {\bf\em weakly increasing tree} on $M$ is a plane tree with $|M| + 1$ nodes that are labeled precisely by all elements in the multiset $M \cup\{0\}$ satisfying
    \begin{enumerate}[(i)]
\item  the labels along a path from the root to any leaf are weakly increasing;
\item  the labels of the children of each node are weakly increasing from right to left. 
\end{enumerate} 
 Denote by $\mathcal{T}_{M}$ the set of weakly increasing  trees on $M$.   
\end{definition}

Note that the root of a weakly increasing tree is always labeled by $0$. See Fig.~\ref{wit} (in left) for a weakly increasing tree on $M=\{1^2,2^4,3^3,4^2\}$. It is clear from the above definition that weakly increasing trees on $[n]:=\{1,2,\ldots,n\}$ are exactly  {\em increasing trees} on $[n]$, while weakly increasing trees on $\{1^n\}$ are in obvious bijection with $\P_n$, the set of all plane trees with $n$ edges. In this fashion, one can study the unity between plane trees and increasing trees in the framework of weakly increasing tees. Since the introduction of this notion by Lin et al.~\cite{Lin20} in 2021, there has been many interesting developments and unexpected applications found in~\cite{DHS,CFu,LI2024,LLY,LLWZ,LM,Lin21}, including connections with the Jacobi elliptic functions and combinatorial interpretation of the $\gamma$-coefficients of the multiset Eulerian polynomials. Remarkably, it was proved in~\cite{Lin20} that 
$$
|\T_M|=\frac{1}{1+N_n}\prod_{i=1}^n{N_i+p_i\choose p_i},
$$
where $N_i:=p_1+p_2+\cdots+p_i$ for each $i\in[n]$.

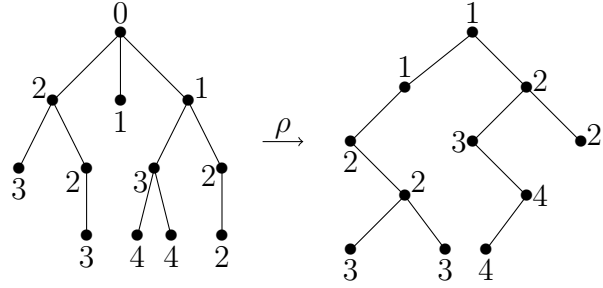
\begin{figure}
\centering
\begin{tikzpicture}[scale=0.9]
\node at (0,4) {$\bullet$}; \node at (0,4.3) {$0$};
\node at (-1,3) {$\bullet$}; \node at (-1.2,3.2) {$2$};
\node at (0,3) {$\bullet$};\node at (0,2.7) {$1$};
\node at (1,3) {$\bullet$};\node at (1.2,3.2) {$1$};
\node at (-1.5,2) {$\bullet$};\node at (-1.5,1.7) {$3$};
\node at (-0.5,2) {$\bullet$};\node at (-0.7,1.8) {$2$};
\node at (0.5,2) {$\bullet$};\node at (0.3,1.8) {$3$};
\node at (1.5,2) {$\bullet$};\node at (1.3,1.9) {$2$};
\node at (-0.5,1) {$\bullet$};\node at (-0.5,0.7) {$3$};
\node at (0.75,1) {$\bullet$};\node at (0.75,0.7) {$4$};
\node at (0.25,1) {$\bullet$};\node at (0.25,0.7) {$4$};

\draw[-] (0,4) -- (-1,3);
\draw[-] (0,4) -- (0,3);
\draw[-] (0,4) -- (1,3);
\draw[-] (-1,3) -- (-1.5,2);
\draw[-] (-1,3) -- (-0.5,2);
\draw[-] (1,3) -- (0.5,2);
\draw[-] (1,3) -- (1.5,2);
\draw[-] (-0.5,2) -- (-0.5,1);
\draw[-] (0.5,2) -- (0.75,1);
\draw[-] (0.5,2) -- (0.25,1);

\draw[-] (1.5,2) -- (1.5,1);
\node at (1.5,0.7) {$2$};
\node at (1.5,1) {$\bullet$};

\draw[->] (2.1,2.4) -- (2.7,2.4);
\node at (2.4,2.6) {$\rho$};

\node at (5.2,4) {$\bullet$};\node at (5.2,4.3) {$1$};
\node at (4.2,3.2) {$\bullet$};\node at (4.2,3.5) {$1$};

\node at (6,3.2) {$\bullet$};\node at (6.2,3.3) {$2$};
\node at (3.4,2.4) {$\bullet$};\node at (3.4,2.1) {$2$};
\node at (5.2,2.4) {$\bullet$};\node at (5,2.4) {$3$};

\node at (3.4,0.8) {$\bullet$};\node at (3.4,0.5) {$3$};

\node at (4.2,1.6) {$\bullet$};\node at (4.4,1.8) {$2$};

\node at (6,1.6) {$\bullet$};\node at (6.2,1.6) {$4$};
\node at (4.8,0.8) {$\bullet$};\node at (4.8,0.5) {$3$};

\node at (5.4,0.8) {$\bullet$};\node at (5.4,0.5) {$4$};

\draw[-] (5.2,4) -- (6,3.2);
\draw[-] (5.2,4) -- (4.2,3.2);
\draw[-] (4.2,3.2) -- (3.4,2.4);
\draw[-] (3.4,2.4) -- (4.2,1.6);
\draw[-] (4.2,1.6) -- (4.8,0.8);
\draw[-] (4.2,1.6) -- (3.4,0.8);
\draw[-] (6,3.2) -- (5.2,2.4);
\draw[-] (5.2,2.4) -- (6,1.6);
\draw[-] (6,1.6) -- (5.4,0.8);
\draw[-] (6,3.2)--(6.8,2.4);
\node at (6.8,2.4) {$\bullet$};\node at (7,2.5) {$2$};

\end{tikzpicture}
\caption{ A weakly increasing tree on $\{1^2,2^4,3^3,4^2\}$ and  its corresponding weakly increasing binary tree under the bijection $\rho$.\label{wit}}
\end{figure}

In a weakly increasing tree, a node without any child is called a {\em leaf}, while a node with children is referred to as an {\em internal node}. Nodes that share the same parent are called {\em siblings}  and those to the left (resp.,~right) of a node $v$  are termed the {\em elder} (resp., {\em younger}) {\em siblings} of $v$. A leaf is considered an {\em old leaf} if it is the leftmost child of its parent, otherwise, it is a {\em young leaf}. Furthermore, an old leaf is called a {\em singleton leaf} if it is the only child of its parent, and an {\em elder leaf} if it has siblings. Finally, an internal node is classified as an {\em elder internal node}  if it is the parent of an elder leaf, a {\em singleton internal node} if it is the parent of a singleton leaf, and a {\em young internal node} otherwise.

For a  tree $T\in \T_{M}$, let $\sleaf(T)$ (resp.,~$\eleaf(T)$, $\sint(T)$, $\eint(T)$, $\yleaf(T)$, $\yint(T)$) denote the number of singleton leaves  (resp.,~elder leaves, singleton internal nodes, young leaves, young internal nodes) in $T$. For example, if $T$ is  the first tree of Fig.~\ref{wit}, then all six statistics on $T$ are equal to $2$. It is clear that
\begin{equation}\label{obs eq1}
\begin{aligned}
     &\sleaf(T)=\sint(T),\quad\eleaf(T)=\eint(T) \quad\text{ and}\\
    &2\sleaf(T)+2\eleaf(T)+\yleaf(T)+\yint(T)=|M|+1.
\end{aligned}
\end{equation}

Over plane trees, old and young leaves  were introduced by Chen, Deutsch and Elizalde~\cite{Chen1} in studying some combinatorial identities, while the further refined six statistics were considered only recently by Dong et al. in~\cite{Dong1}. 

\begin{theorem}\label{Thm:sym1}
   Fix a multiset $M$ with $|M|\geq2$. There exists an involution $\Phi$ on $\T_M$ that transforms the quadruple $(\sleaf,\eleaf,\yleaf,\yint)$ to $(\eleaf,\sleaf,\yint,\yleaf)$.
\end{theorem}

For $n\geq2$, the equidistribution of the two quadruples  
    \begin{equation*}
        (\sleaf,\eleaf,\yleaf,\yint) \text{~~~and~~~} (\eleaf,\sleaf,\yint,\yleaf)
    \end{equation*}
on $\P_n$ was proved algebraically by Dong et al.~\cite{Dong1}. Theorem~\ref{Thm:sym1} extends their equidistribution from plane trees to weakly increasing trees.

\begin{definition}[Tip-augmented weakly increasing trees]
   A  tree in $\T_M$ without any young internal node is called a {\bf\em tip-augmented weakly increasing tree} on $M$.
 The set of all tip-augmented weakly increasing trees on $M$ is denoted by $\mathcal{A}_{M}$. Note that tip-augmented plane trees were introduced by Donaghey~\cite{Donaghey1}.
\end{definition}

For a tree $T\in\T_M$, let  $\oleaf(T):=\sleaf(T)+\eleaf(T)$ and $\oint(T):=\sint(T)+\eint(T)$.
\begin{theorem}
\label{Thm:tip}
    For any multiset $M$,  we have
    \begin{equation}\label{eq:tip}
    \sum_{T\in \T_M}x_1^{\oleaf(T)}x_2^{\yleaf(T)}y_1^{\oint(T)}y_2^{\yint(T)}=\sum_{T\in \A_{M}}(x_1y_1)^{\oleaf(T)}(x_2+y_2)^{\yleaf(T)}.
\end{equation}
\end{theorem}

For $n\geq0$, Chen and Pan~\cite{Chen2} considered the following two-variable Motzkin polynomials
\begin{equation*}
    M_n(u,v)=\sum_{k=0}^{\left\lfloor n/2\right\rfloor}\binom{n}{2 k} C_k u^{k+1} v^{n-2 k},
\end{equation*}
where $C_n:=\frac{1}{1+n}{2n\choose n}$ is the $n$-th {\em Catalan number}. When $u$ and $v$ are positive integers, the Motzkin polynomial $M_n(u, v)$ was termed the {\em generalized Motzkin numbers} by Sun~\cite{Sun1}. We shall provide a group action proof of Theorem~\ref{Thm:tip} and derive from its plane tree case  the following relationship between refined Narayana polynomials and Motzkin polynomials.

\begin{corollary}[Dong-Du-Ji-Zhang~\cite{Dong1}]
\label{thm:motz}
For $n\geq1$, we have
 \begin{equation}\label{rel:motz}
    \sum_{T\in \P_n}x_1^{\oleaf(T)}x_2^{\yleaf(T)}y_1^{\oint(T)}y_2^{\yint(T)}=M_{n-1}(x_1y_1,x_2+y_2).
\end{equation}
\end{corollary}

Specializing $x_1=x_2=x$ and $y_1=y_2=y$ in~\eqref{rel:motz} leads to a result of Chen and Pan~\cite[Theorem~1.3~(1.14)]{Chen2}. 
In order to prove~\eqref{rel:motz} using Chen's context-free grammar, Dong et al.~\cite{Dong1} considered  further refinements of elder leaves and young leaves on plane trees. We extend their notions to weakly increasing trees and introduce new refinements of  singleton leaves. In a weakly increasing tree, we define six kinds of leaves:
\begin{itemize}
\item a singleton leaf  whose parent has elder siblings is called a {\em singleton leaf with uncle};
\item a singleton leaf  whose parent has no elder sibling is called a {\em singleton leaf with no uncle};
\item an elder leaf is considered  as an {\em elder twin leaf} if  the second child of its parent is also a leaf;
\item an elder leaf is considered as an {\em elder non-twin leaf} if  the second child of its parent is not a leaf;
\item a young leaf is considered as a {\em second leaf} if it is the second child of its parent;
\item a young leaf is considered as a {\em younger leaf} if it is not the second child of its parent.
\end{itemize}
For $T\in\T_M$, let $\suleaf(T)$  (resp.,~$\snuleaf(T)$, $\etleaf(T)$, $\entleaf(T)$, $\syleaf(T)$, $\yerleaf(T)$) denote the number of singleton leaves  with uncle (resp., singleton leaves with no uncle,  elder twin leaves, elder non-twin leaves, second leaves and younger leaves) in $T$.

Let $\A_n$ be the set of all tip-augmented plane trees in $\P_n$, i.e., $\A_n=\A_{\{1^n\}}$. Consider a refinement of the Motzkin polynomial $M_n(u,v)$ as 
    $$
    M_n(u_1,u_2,u_3;v_1,v_2):=\sum_{T\in \A_{n+1}}u_1^{\sleaf(T)} u_2^{\etleaf(T)} u_3^{\entleaf(T)} v_1^{\yerleaf(T)} v_2^{\syleaf(T)}.
    $$
Dong et al.~\cite{Dong1} proved algebraically the following symmetry on tip-augmented plane trees and asked for a bijective proof. 
\begin{proposition}[Dong-Du-Ji-Zhang~\cite{Dong1}]
 For $n\geq2$, we have 
\begin{equation}\label{eq:sym:mot}
M_n(u_1,u_2,u_3;v_1,v_2)=M_n(u_3,u_2,u_1;v_1,v_2).
\end{equation}
\end{proposition}

We answer their bijective problem by proving the following generalization. 
\begin{theorem}\label{Thm:sym2}
   Fix a multiset $M$ with $|M|\geq2$. There exists an involution $\Psi$ on $\T_M$ that preserves the quadruple $(\snuleaf, \etleaf,\syleaf,\yerleaf)$ but exchanges the pair $(\suleaf,\entleaf)$.  Moreover, $\Psi$ preserves the statistic ``$\yint$'' and so $\Psi$ restricts to an involution on $\A_M$ which proves the symmetry~\eqref{eq:sym:mot}. 
\end{theorem}

It turns out that  the proofs of Theorems~\ref{Thm:sym1},~\ref{Thm:tip} and~\ref{Thm:sym2} are all based on a simple switching operation on binary trees, namely switching left and right branches of a node. Moreover, this operation enables us to built non-recursively a bijection due to Deutsch~\cite{Deut} on plane trees and independently extended to weakly increasing trees by Lin--Ma~\cite{LM} and Chen--Fu~\cite{CFu}. As a by product, we find that the statistic of ``internal nodes on even-levels'' has the same distribution as ``$\oleaf$'' over plane trees (see Corollary~\ref{cor:elintleaf}). This paper reflects the principle: whenever you encounter constructing bijections on weakly increasing trees, then expect that binary trees  will facilitate your problem. 

Applications of Theorems~\ref{Thm:sym1},~\ref{Thm:tip} and~\ref{Thm:sym2} to permutation patterns and statistics will also be discussed, including a generalization of a symmetry regarding ``maximal number of non-overlapping descents/ascents'' on $231$-avoiding permutations due to Kitaev and Zhang~\cite{KZ}. 

The rest of this paper is organized as follows. In Section~\ref{sec:2},  we prove Theorems~\ref{Thm:sym1},~\ref{Thm:tip} and~\ref{Thm:sym2} and provide a non-recursive construction of  Deutsch's bijection using a simple switching operation on binary trees.  Using binary trees, we can also prove a unified formula that generalizes  two main formulas of Dong et al.~\cite{Dong1}. Section~\ref{sec:3} is devoted to some applications of our bijections  to permutation patterns and statistics, including two symmetries in $312$-avoiding permutations, one of which generalizes a result of Kitaev and Zhang~\cite{KZ}.

\section{Combinatorics of weakly increasing trees via binary trees} 
\label{sec:2}
In this section, via the switching  operation on binary trees, we prove Theorems~\ref{Thm:sym1},~\ref{Thm:tip} and~\ref{Thm:sym2} and provide a non-recursive construction of a bijection on plane trees due to Deutsch. The generating function for the quintuple $(\sleaf,\etleaf, \entleaf, \syleaf, \yerleaf)$ over plane trees will also be computed via binary trees, unifying two generating function formulas due to Dong et al.~\cite{Dong1}. 

\subsection{Weakly increasing binary trees}

We begin by introducing  weakly increasing binary trees, which are in natural bijection with weakly increasing trees. A {\em binary tree} is a rooted tree where each internal node has either a left child, a right child, or both. Let $\B_n$ denote the set of all binary trees with $n$ nodes.
As a  common generalization of binary trees and binary increasing trees, the weakly increasing binary trees was first used in~\cite{LM}. 

\begin{definition}[Weakly increasing binary trees~\cite{LM}]
 For a fixed  multiset $M$, a  {\bf\em weakly increasing binary  tree} on $M$ is a labeled  binary tree  such that 
 \begin{enumerate}[(i)]
\item  the labels of the nodes form precisely the multiset $M$ and
\item  the labels along a path from the root to any leaf are weakly increasing. 
\end{enumerate}   
 Denote by $\mathcal{B}_{M}$ the set of weakly increasing binary trees on $M$. 
\end{definition}

See Fig.~\ref{wit} (in right) for a weakly increasing binary tree on $M=\{1^2,2^4,3^3,4^2\}$. Note that weakly increasing binary trees on $[n]$ are precisely increasing binary trees on $[n]$, while weakly increasing binary trees on $\{1^n\}$ are in  bijection with $\B_n$. 
 For a weakly increasing tree $T\in\T_M$, we define its corresponding weakly increasing binary tree $\rho(T)\in\B_M$ by imposing the condition on each pair of non-root nodes $(x,y)$ in $T$: 
 \begin{enumerate}[(i)]
 \item $y$ is the right child of $x$ in $\rho(T)$ only if  when  $y$ is the rightmost child of  $x$ in $T$;
 \item  $y$ is the left child of $x$ in $\rho(T)$ only if when $x$ is the closest elder sibling of $y$ in $T$.
 \end{enumerate}
The mapping $T\mapsto\rho(T)$ establishes a natural one-to-one correspondence between $\T_M$ and $\B_M$, which extends the classical  bijection between $\P_n$ and $\B_n$ in~\cite[Page~9]{St1}. See Fig.~\ref{wit} for an example of the bijection $\rho$. 

For a binary tree $B\in\B_M$ and a node $x$ of $B$, let $\phi_x(B)$ be the tree in $\B_M$ obtained from $B$ by switching  the left branch and the right branch of $x$, if any. For any two nodes $x,y$ of $B$, it is clear that $\phi_x\circ\phi_y(T)=\phi_y\circ\phi_x(T)$, i.e., $\phi_x$ and $\phi_y$ are commute. Thus, if $A$ is a set of some specified nodes of $B$, then it is reasonable to define the transformation $B\mapsto\prod_{x\in A}\phi_x(B)$.

Throughout this paper, if $\st$ is a statistic on $\T_M$ and $B$ is a binary tree in $\B_M$, then we define 
$$\st(B)=\st(\rho^{-1}(B)),$$ 
for the sake of convenience. For example, if $B\in\B_M$ with $|M|\ge2$, then $\sleaf(B)$ is the number of right leaves of $B$ by Lemma~\ref{lem:phi}~$(i)$. 

\subsection{Proofs of Theorems~\ref{Thm:sym1} and~\ref{Thm:sym2}}

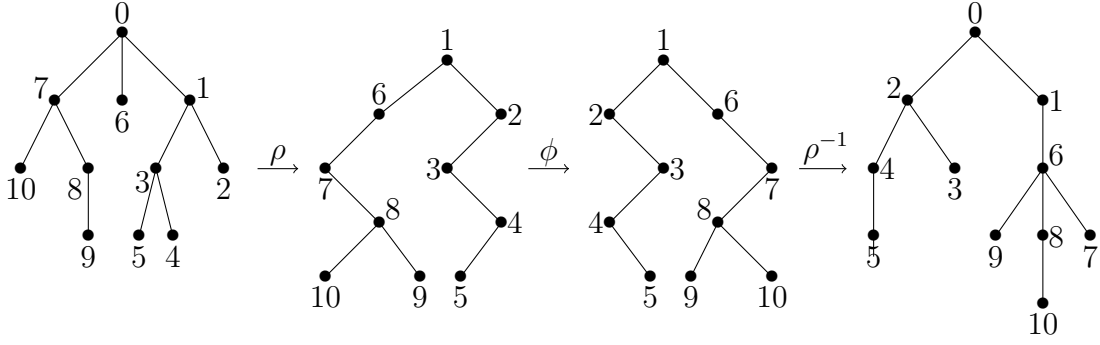
\begin{figure}
\centering

\begin{tikzpicture}[scale=0.9]

\node at (0,4) {$\bullet$}; \node at (0,4.3) {$0$};
\node at (-1,3) {$\bullet$}; \node at (-1.2,3.2) {$7$};
\node at (0,3) {$\bullet$};\node at (0,2.7) {$6$};
\node at (1,3) {$\bullet$};\node at (1.2,3.2) {$1$};
\node at (-1.5,2) {$\bullet$};\node at (-1.5,1.7) {$10$};
\node at (-0.5,2) {$\bullet$};\node at (-0.7,1.7) {$8$};
\node at (0.5,2) {$\bullet$};\node at (0.3,1.8) {$3$};
\node at (1.5,2) {$\bullet$};\node at (1.5,1.7) {$2$};
\node at (-0.5,1) {$\bullet$};\node at (-0.5,0.7) {$9$};
\node at (0.75,1) {$\bullet$};\node at (0.75,0.7) {$4$};
\node at (0.25,1) {$\bullet$};\node at (0.25,0.7) {$5$};

\draw[-] (0,4) -- (-1,3);
\draw[-] (0,4) -- (0,3);
\draw[-] (0,4) -- (1,3);
\draw[-] (-1,3) -- (-1.5,2);
\draw[-] (-1,3) -- (-0.5,2);
\draw[-] (1,3) -- (0.5,2);
\draw[-] (1,3) -- (1.5,2);
\draw[-] (-0.5,2) -- (-0.5,1);
\draw[-] (0.5,2) -- (0.75,1);
\draw[-] (0.5,2) -- (0.25,1);

\draw[->] (2,2) -- (2.6,2);
\node at (2.3,2.2) {$\rho$};

\node at (4.8,3.6) {$\bullet$};\node at (4.8,3.9) {$1$};
\node at (3.8,2.8) {$\bullet$};\node at (3.8,3.1) {$6$};

\node at (5.6,2.8) {$\bullet$};\node at (5.8,2.8) {$2$};
\node at (3,2) {$\bullet$};\node at (3,1.7) {$7$};
\node at (4.8,2) {$\bullet$};\node at (4.6,2) {$3$};

\node at (3,0.4) {$\bullet$};\node at (3,0.1) {$10$};

\node at (3.8,1.2) {$\bullet$};\node at (4,1.4) {$8$};

\node at (5.6,1.2) {$\bullet$};\node at (5.8,1.2) {$4$};
\node at (4.4,0.4) {$\bullet$};\node at (4.4,0.1) {$9$};

\node at (5,0.4) {$\bullet$};\node at (5,0.1) {$5$};

\draw[-] (4.8,3.6) -- (5.6,2.8);
\draw[-] (4.8,3.6) -- (3.8,2.8);
\draw[-] (3.8,2.8) -- (3,2);
\draw[-] (3,2) -- (3.8,1.2);
\draw[-] (3.8,1.2) -- (4.4,0.4);
\draw[-] (3.8,1.2) -- (3,0.4);
\draw[-] (5.6,2.8) -- (4.8,2);
\draw[-] (4.8,2) -- (5.6,1.2);
\draw[-] (5.6,1.2) -- (5,0.4);

\draw[->] (6,2) -- (6.6,2);
\node at (6.3,2.25) {$\phi$};

\node at (8,3.6) {$\bullet$};\node at (8,3.9) {$1$};
\node at (7.2,2.8) {$\bullet$};\node at (7,2.8) {$2$};
\node at (8.8,2.8) {$\bullet$};\node at (9,3) {$6$};
\node at (7.2,1.2) {$\bullet$};\node at (7,1.2) {$4$};

\node at (8,2) {$\bullet$};\node at (8.2,2) {$3$};
\node at (7.8,0.4) {$\bullet$};\node at (7.8,0.1) {$5$};
\node at (8.8,1.2) {$\bullet$};\node at (8.6,1.4) {$8$};
\node at (8.4,0.4) {$\bullet$};\node at (8.4,0.1) {$9$};
\node at (9.6,0.4) {$\bullet$};\node at (9.6,0.1) {$10$};
\node at (9.6,2) {$\bullet$};\node at (9.6,1.7) {$7$};

\draw[-] (8,3.6) -- (8.8,2.8);
\draw[-] (8,3.6) -- (7.2,2.8);
\draw[-] (7.2,2.8) -- (8,2);
\draw[-] (8,2) -- (7.2,1.2);
\draw[-] (7.2,1.2) -- (7.8,0.4);
\draw[-] (8.8,2.8) -- (9.6,2);
\draw[-] (9.6,2) -- (8.8,1.2);
\draw[-] (8.8,1.2) -- (8.4,0.4);
\draw[-] (8.8,1.2) -- (9.6,0.4);

\draw[->] (10,2) -- (10.7,2);
\node at (10.4,2.3) {$\rho^{-1}$};

\node at (12.6,4) {$\bullet$}; \node at (12.6,4.3) {$0$};
\node at (11.6,3) {$\bullet$}; \node at (11.4,3.2) {$2$};
\node at (13.6,3) {$\bullet$};\node at (13.8,3) {$1$};
\node at (11.1,2) {$\bullet$};\node at (11.3,2) {$4$};
\node at (12.3,2) {$\bullet$};\node at (12.3,1.7) {$3$};
\node at (13.6,2) {$\bullet$};\node at (13.8,2.2) {$6$};
\node at (11.1,1) {$\bullet$};\node at (11.1,0.7) {$5$};
\node at (12.9,1) {$\bullet$};\node at (12.9,0.7) {$9$};
\node at (13.6,1) {$\bullet$};\node at (13.8,1) {$8$};
\node at (14.3,1) {$\bullet$};\node at (14.3,0.7) {$7$};
\node at (13.6,0) {$\bullet$};\node at (13.6,-0.3) {$10$};

\draw[-] (12.6,4) -- (11.6,3);
\draw[-] (12.6,4) -- (13.6,3);
\draw[-] (11.6,3) -- (11.1,2);
\draw[-] (11.1,2) -- (11.1,0.7);
\draw[-] (11.6,3) -- (12.3,2);
\draw[-] (13.6,3) -- (13.6,2);
\draw[-] (13.6,2) -- (12.9,1);
\draw[-] (13.6,2) -- (13.6,1);
\draw[-] (13.6,2) -- (14.3,1);
\draw[-] (13.6,1) -- (13.6,0);

\end{tikzpicture}
\caption{An example of the involution $\Phi=\rho^{-1}\circ\phi\circ\rho$.\label{invo1}}
\end{figure}

The proof of Theorem~\ref{Thm:sym1} is based on the following key observation. 
\begin{lemma}\label{lem:phi}
Let $T$ be a weakly increasing tree with at least two edges. Then, 
\begin{enumerate}[(i)]
    \item $\sleaf(T)$ equals the number of right leaves in  $\rho(T)$;
    \item $\eleaf(T)$ equals the number of left leaves in  $\rho(T)$;
    \item $\yleaf(T)$ equals the number of nodes having only left child in  $\rho(T)$;
    \item $\yint(T)$ equals the  number of nodes having only right child in  $\rho(T)$.
\end{enumerate}
\end{lemma}

\begin{proof}
    We will only show $(iv)$, as the other three assertions follow easily from the construction of $\rho$. Note that a node is a  young internal node if and only if its eldest child is an internal node.  Since the map $\rho$ assigns an internal eldest node of $T$ to a node in $\rho(T)$ that lacks a left child but possesses right child, $(iv)$ follows.
\end{proof}

\begin{proof}[Proof of Theorem~\ref{Thm:sym1}] For $B\in\B_M$, let $\phi(B)$ be the minor symmetry of $T$, i.e., $\phi(B)=\prod_x\phi_x(B)$, running  over all nodes $x$ of $B$. Then the composition $\Phi=\rho^{-1}\circ\phi\circ\rho$ is an involution with the desired property in view of Lemma~\ref{lem:phi}. See Fig.~\ref{invo1} for an example of the involution $\Phi$. 
\end{proof}

\begin{remark}
A slight  modification of the involution $\Phi$ was used in~\cite{LI2024} to study other symmetry of trees.
\end{remark}

\begin{lemma}\label{obv:rho2}
    Let $T$ be a weakly increasing tree with at least two edges. Then, 
    \begin{enumerate}[(i)]
        \item $\suleaf(T)$ equals the number of right leaves in $\rho(T)$ whose parent has  left child;
        \item $\snuleaf(T)$ equals the number of right leaves in $\rho(T)$ whose parent has no left child;
        \item $\etleaf(T)$ equals the number of left leaves in $\rho(T)$ whose parent has no right child;
        \item $\entleaf(T)$ equals the number of left leaves in $\rho(T)$ whose parent has  right child;
        \item $\yerleaf(T)$ equals the number of nodes in $\rho(T)$ with only left child and whose left child has left child;
        \item $\syleaf(T)$ equals the number of nodes in $\rho(T)$ with only left child and whose left child has no left child.
    \end{enumerate}
\end{lemma}
\begin{proof}
    All the assertions can be verified routinely from the construction of $\rho$. We leave the details to the reader. 
\end{proof}

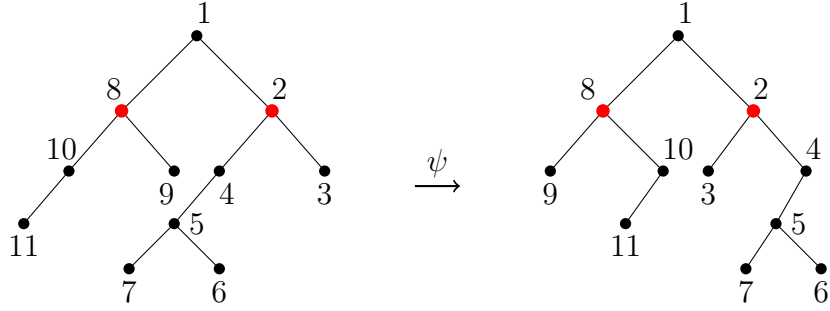
\begin{figure}
    \centering
    \begin{tikzpicture}

        \draw[-] (0,4) -- (-1,3);
        \draw[-] (0,4) -- (1,3);
        \draw[-] (-1,3) -- (-1.7,2.2);
        \draw[-] (-1.7,2.2) -- (-2.3,1.5);
        \draw[-] (-1,3) -- (-0.3,2.2);
        \draw[-] (1,3) -- (0.3,2.2);
        \draw[-] (1,3) -- (1.7,2.2);
        \draw[-] (0.3,2.2) -- (-0.3,1.5);
        \draw[-] (-0.3,1.5) -- (-0.9,0.9);
        \draw[-] (-0.3,1.5) -- (0.3,0.9);

        \node at (0,4) [circle,fill,inner sep=1.5pt]{};\node at (0.1,4.3) {$1$};
        \node at (-1,3) [circle,fill=red,inner sep=1.8pt]{};\node at (-1.1,3.3) {$8$};
        \node at (1,3) [circle,fill=red,inner sep=1.8pt]{};\node at (1.1,3.3) {$2$};
        \node at (-1.7,2.2) [circle,fill,inner sep=1.5pt]{};\node at (-1.8,2.5) {$10$};
        \node at (-0.3,2.2) [circle,fill,inner sep=1.5pt]{};\node at (-0.4,1.9) {$9$};
        \node at (0.3,2.2) [circle,fill,inner sep=1.5pt]{};\node at (0.4,1.9) {$4$};
        \node at (1.7,2.2) [circle,fill,inner sep=1.5pt]{};\node at (1.7,1.9) {$3$};
        \node at (-2.3,1.5) [circle,fill,inner sep=1.5pt]{};\node at (-2.3,1.2) {$11$};
        \node at (-0.3,1.5) [circle,fill,inner sep=1.5pt]{};\node at (0,1.5) {$5$};
        \node at (-0.9,0.9) [circle,fill,inner sep=1.5pt]{};\node at (-0.9,0.6) {$7$};
        \node at (0.3,0.9) [circle,fill,inner sep=1.5pt]{};\node at (0.3,0.6) {$6$};

        \draw[->, thick] (2.9,2) -- (3.5,2);
        \node at (3.2,2.3) {$\psi$};

        \draw[-] (6.4,4) -- (5.4,3);
        \draw[-] (6.4,4) -- (7.4,3);
        \draw[-] (5.4,3) -- (4.7,2.2);
        \draw[-] (5.4,3) -- (6.2,2.2);
        \draw[-] (7.4,3) -- (6.8,2.2);
        \draw[-] (7.4,3) -- (8.1,2.2);
        \draw[-] (6.2,2.2) -- (5.7,1.5);
        \draw[-] (8.1,2.2) -- (7.7,1.5);
        \draw[-] (7.7,1.5) -- (7.3,0.9);
        \draw[-] (7.7,1.5) -- (8.3,0.9);

        \node at (4.7,2.2) [circle,fill,inner sep=1.5pt]{};\node at (4.7,1.9) {$9$};
        \node at (6.2,2.2) [circle,fill,inner sep=1.5pt]{};\node at (6.4,2.5) {$10$};
        \node at (5.4,3) [circle,fill=red,inner sep=1.8pt]{};\node at (5.2,3.3) {$8$};
        \node at (6.4,4) [circle,fill,inner sep=1.5pt]{};\node at (6.5,4.3) {$1$};
        \node at (7.4,3) [circle,fill=red,inner sep=1.8pt]{};\node at (7.5,3.3) {$2$};
        \node at (6.8,2.2) [circle,fill,inner sep=1.5pt]{};\node at (6.8,1.9) {$3$};
        \node at (8.1,2.2) [circle,fill,inner sep=1.5pt]{};\node at (8.2,2.5) {$4$};
        \node at (5.7,1.5) [circle,fill,inner sep=1.5pt]{};\node at (5.7,1.2) {$11$};
        \node at (7.7,1.5) [circle,fill,inner sep=1.5pt]{};\node at (8,1.5) {$5$};
        \node at (7.3,0.9) [circle,fill,inner sep=1.5pt]{};\node at (7.3,0.6) {$7$};
        \node at (8.3,0.9) [circle,fill,inner sep=1.5pt]{};\node at (8.3,0.6) {$6$};

    \end{tikzpicture}
    \caption{An example of the involution $\psi$.}
    \label{invo3}
\end{figure}

\begin{proof}[Proof of Theorem~\ref{Thm:sym2}] For a binary tree $B\in\B_M$ and $x$  a node of $B$, then $x$ is said to be {\em unbalanced} if one of its child is a leaf but the other child is an internal node. For example, the unbalanced nodes of the binary tree in  Fig.~\ref{invo3} (in left) are $2$ and $8$. Let $\Unb(B)$ be the set of all unbalanced nodes of $B$. Let 
$$
\psi(B):=\prod_{x\in\Unb(T)}\phi_x(B).
$$
It is plain to see that $\Unb(B)=\Unb(\psi(B))$ and so $\psi$ is an involution on $\B_M$. See Fig.~\ref{invo3} for an example of $\psi$, where unbalanced nodes are marked in red. Set $\Psi=\rho^{-1}\circ\psi\circ\rho$. As $\psi$ is an involution on $\B_M$,  $\Psi$ is an involution on $\T_M$.

In order to prove the first statement, we need to show that 
\begin{itemize}
\item[(a)] $\psi$ preserves the quadruple $(\snuleaf, \etleaf,\syleaf,\yerleaf)$;
\item[(b)] for any $B\in\B_M$, $\suleaf(B)=\entleaf(\psi(B))$. 
\end{itemize}
Here  the six statistics  $(\snuleaf, \etleaf,\syleaf,\yerleaf, \suleaf,\entleaf)$ over binary trees have their meanings as proved in Lemma~\ref{obv:rho2}. Assertion (a) follows from the fact that 
$$
(\snuleaf, \etleaf,\syleaf,\yerleaf)B=(\snuleaf, \etleaf,\syleaf,\yerleaf)\phi_x(B)
$$
whenever  $x$ is an unbalanced node of $B$. To see assertion (b), we note that 
\begin{align*}
\suleaf(B)&=|\{x\in\Unb(B): \text{ right child of $x$ is a leaf}\}|+\twin(B),\\
\entleaf(B)&=|\{x\in\Unb(B): \text{ left child of $x$ is a leaf}\}|+\twin(B),
\end{align*}
where $\twin(B)$ is the number of nodes of $B$ whose two children are leaves. Assertion (b) now follows from the fact that if $x$ is an unbalanced node whose right child is a leaf in $B$, then $x$ becomes an unbalanced node whose left child is a leaf in $\psi(B)$. This completes the proof of the first statement. 

 By Lemma~\ref{lem:phi}~$(iv)$, $\yint(B)$ is the  number of nodes having only right child in $B$ for any $B\in\B_M$. Thus, $\psi$ preserves the statistic ``$\yint$'', completing the proof of the theorem.
\end{proof}

\subsection{A group action proof of Theorem~\ref{Thm:tip} and a related bijective problem of Dong et al.}

We use the preorder~\cite[Page~10]{St1} (i.e., recursively traversing the parent to the left subtree then to the right subtree) to order the nodes of binary trees. Fix a multiset $M$ with $|M|=m$. For $B\in\B_M$ and $i\in[m]$, suppose that  $x$ is the $i$-th (under preorder) node of $B$ and define 
$$
\varphi_i(B)=
\begin{cases}
\,\,\phi_x(B),\quad&\text{if $x$ has exactly one child in $B$};\\
\,\,B,&\text{otherwise}. 
\end{cases}
$$
It is clear that the order of the nodes in $B$ remains the same as in $\varphi_i(B)$ and so $\varphi^2_i(B)=B$. As $\varphi_i$ and $\varphi_j$ commute for all $i,j\in[m]$, for any subset $S\subseteq[m]$, we can define the function $\varphi_S: \B_M\rightarrow\B_M$ by
$$
\varphi_S(B)=\prod_{i\in S}\varphi_i(B).
$$
The group $\Z_2^m$ acts on $\B_M$ via the functions $\varphi_S$,  $S\subseteq[m]$. By Lemma~\ref{lem:phi}, we have 
\begin{equation}\label{eq:action2}
\sum_{T\in \T_M}x_1^{\oleaf(T)}x_2^{\yleaf(T)}y_1^{\oint(T)}y_2^{\yint(T)}=\sum_{B\in\B_M} (x_1y_1)^{\oleaf(B)}x_2^{\yleaf(B)}y_2^{\yint(B)},
\end{equation}
where $\oleaf(B)$ is the number of leaves in $B$, $\yleaf(B)$ (resp.,~$\yint(B)$) is the number of nodes having only left (resp.,~right) child in $B$. The following property of $\varphi_i$ is obvious from its definition. 

\begin{lemma}\label{lem:action}
For $B\in\B_M$ and $i\in[m]$, 
\begin{itemize}
\item[(i)] $\oleaf(B)=\oleaf(\varphi_i(B))$;
\item[(ii)] the $i$-th node of $B$ has only left (resp.,~right) child  if and only if the $i$-th node of $\varphi_i(B)$ has only right (resp.,~left) child. 
\end{itemize}
\end{lemma}

We are ready for the proof of  Theorem~\ref{Thm:tip}. 
\begin{proof}[Proof of Theorem~\ref{Thm:tip}]
For any $B\in\B_M$, let $[B]:=\{g(B): g\in\Z_2^m\}$ be the orbit of $B$ under the $\Z_2^m$-action. By Lemma~\ref{lem:action}, each orbit $[B]$ contains a unique binary tree $\tilde B$ with $\yint(\tilde B)=0$. Thus, by Lemma~\ref{lem:action} we have
$$
\sum_{B\in[\tilde B]} (x_1y_1)^{\oleaf(B)}x_2^{\yleaf(B)}y_2^{\yint(B)}=(x_1y_1)^{\oleaf(\tilde B)}(x_2+y_2)^{\yleaf(\tilde B)}
$$
for any $\tilde B\in\B_M$ with $\yint(\tilde B)=0$. 
Summing over all orbits yields 
$$
\sum_{B\in\B_M} (x_1y_1)^{\oleaf(B)}x_2^{\yleaf(B)}y_2^{\yint(B)}=\sum_{\tilde B\in\B_M\atop\yint(\tilde B)=0}(x_1y_1)^{\oleaf(\tilde B)}(x_2+y_2)^{\yleaf(\tilde B)}.
$$
Combining with~\eqref{eq:action2} gives~\eqref{eq:tip}. 
\end{proof}

\begin{remark}
Under the bijection $\rho$, the $\Z_2^m$ group action used  above is essential the same as the generalized Foata--Strehl action on $\T_M$ introduced in~\cite[Section~7]{Lin20}. 
\end{remark}

 Let 
$$
\P_e(M,k):=\{T\in\T_M: \oleaf(T)=k, \text{$\yleaf(T)$ is even}\}
$$
and
$$
\P_o(M,k):=\{T\in\T_M: \oleaf(T)=k, \text{$\yleaf(T)$ is odd}\}.
$$
As
\begin{equation}\label{eq:all}
2\oleaf(T)+\yleaf(T)+\yint(T)=|M|+1
\end{equation}
for any $T\in\T_M$,
Setting $x_2=-1$ and $y_1=y_2=1$ in~\eqref{eq:tip} results in the following combinatorial assertion. 
\begin{corollary}\label{wit:eo}
Fix a multiset $M$ with $|M|=m$. Then, for $1\leq k\leq \lfloor\frac{m}{2}\rfloor$:
\begin{equation}\label{eqwit:eo}
|\P_e(M,k)|=|\P_o(M,k)|.
\end{equation}
\end{corollary}

\begin{remark}
Corollary~\ref{wit:eo} for plane trees (i.e., $M=\{1^m\}$) was derived from~\eqref{rel:motz}
by Dong et al.~\cite{Dong1}, who also asked for a bijective proof. A bijective  proof of~\eqref{eqwit:eo} can be easily constructed as follows. By~\eqref{eq:all} and the condition  $1\leq k\leq \lfloor\frac{m}{2}\rfloor$, we deduce that $\yleaf(T)+\yint(T)>0$ for any $T\in\P_e(M,k)$. Given $T\in\P_e(M,k)$, find the first node (under preorder), say $x$,   having only one child in $\rho(T)$ and then set $T'=\rho^{-1}\phi_x(\rho(T))$. The map $T\mapsto T'$ sets up an one-to-one correspondence between $\P_e(M,k)$ and $\P_o(M,k)$.
\end{remark}

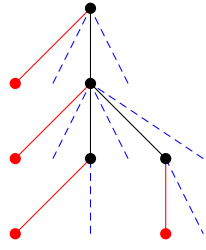
\begin{figure}
    \centering
    \begin{tikzpicture}
    \draw[densely dashed,blue] (0,0)--(0,1);   
    \draw[densely dashed,blue] (0.5,1)--(0,2);
    \draw[densely dashed,blue] (-0.5,1)--(0,2);
    \draw[densely dashed,blue] (1.5,1)--(0,2);
    \draw[densely dashed,blue] (-0.5,2)--(0,3);
    \draw[densely dashed,blue] (0.5,2)--(0,3);
    \draw[densely dashed,blue] (1,1)--(1.5,0);

    \draw[-]  (0,3)--(0,2);
    \draw[-]  (0,2)--(0,1);
    \draw[-]  (0,2)--(1,1);
    \draw[-,red]  (0,3)--(-1,2);
    \draw[-,red]  (0,2)--(-1,1);
    \draw[-,red]  (0,1)--(-1,0);
    \draw[-,red]  (1,1)--(1,0);

     \node at (0,3) [circle,fill,inner sep=1.5pt]{}; 
    \node at (0,2) [circle,fill,inner sep=1.5pt]{};
    \node at (0,1) [circle,fill,inner sep=1.5pt]{};
    
    \node at (-1,2) [circle,fill=red,inner sep=1.5pt]{};
    \node at (-1,1) [circle,fill=red,inner sep=1.5pt]{};
    \node at (-1,0) [circle,fill=red,inner sep=1.5pt]{};
    \node at (1,1) [circle,fill,inner sep=1.5pt]{};
    \node at (1,0) [circle,fill=red,inner sep=1.5pt]{};
\end{tikzpicture}

    \caption{An example of constructing a tip-augmented plane tree.}
    \label{fig:tip-augment}
\end{figure}

\begin{proof}[Proof of Corollary~\ref{thm:motz}]
By the plane tree case of Theorem~\ref{Thm:tip}, all we need to prove is 
\begin{equation}\label{def two-v}
    \sum_{T\in \A_{n+1}}u^{\oleaf(T)}v^{\yleaf(T)}=\sum_{k=0}^{\left\lfloor n/2\right\rfloor}\binom{n}{2 k} C_k u^{k+1} v^{n-2 k}.
\end{equation}
To form a tip-augmented plane tree with $k+1$ old leaves, we can first choose a plane trees with $k$ edges, append one node, as old leaf, to each of its $k+1$ nodes, and finally  insert $n-2k$ nodes, as young leaves, immediately to the right of some of the  $2k+1$ edges (allowing for repetition). See the illustration in Fig.~\ref{fig:tip-augment} for $k=3$. There are $C_k$ ways to choose a plane tree and $\left(\binom{2k+1}{n-2k}\right)=\binom{n}{2 k}$ ways to choose $n-2k$ positions to insert nodes in the final step, which proves~\eqref{def two-v}, as desired.  
\end{proof}

 It is well known (see~\cite{Chen1990}) that the number of plane trees with $n$ edges and $k$ leaves is the {\em Narayana number} $N_{n,k}=\frac{1}{n}{n\choose k}{n\choose k-1}$.  One consequence of Corollary~\ref{thm:motz} is the following explicit formula for the refined Narayana numbers
 $$
N_{n,k,l}:=|\{T\in\P_n: \oleaf(T)=k,\yleaf(T)=l\}|
 $$ 
 due to Chen, Deutsch and Elizalde~\cite{Chen1}. 

\begin{corollary}[Chen-Deutsch-Elizalde~\cite{Chen1}]
For $n,k\geq1$ and $l\geq0$, we have 
$$
N_{n,k,l}=\frac{1}{n}{n\choose k}{n-k\choose l}{n-k-l\choose k-1}. 
$$
\end{corollary}
\begin{proof}
Setting $y_1=y_2=1$ in~\eqref{rel:motz} gives
$$
\sum_{T\in\P_n}x_1^{\oleaf(T)}x_2^{\yleaf}=\sum_{k=0}^{\left\lfloor (n-1)/2\right\rfloor}\binom{n-1}{2 k} C_k x_1^{k+1} (1+x_2)^{n-1-2 k}. 
$$
Extracting the coefficient of $x_1^kx_2^l$ yields
$$
N_{n,k,l}={n-1\choose 2k-2}C_{k-1}{n+1-2k\choose l}=\frac{1}{n}{n\choose k}{n-k\choose l}{n-k-l\choose k-1},
$$
as desired. 
\end{proof}

\subsection{Generating function for $(\sleaf,\etleaf,\entleaf,\yerleaf,\syleaf)$ on plane trees}
Consider the refined Narayana polynomials 
$$
G_n(x_1,x_2,x_3,x_4):=\sum_{T\in \P_n}x_1^{\sleaf(T)}x_2^{\eleaf(T)}x_3^{\yleaf(T)}x_4^{\yint(T)}\quad\text{for $n\geq2$}
$$
and $G_1(x_1,x_2,x_3,x_4)=x_2$. The main results of~\cite{Dong1} are two formulas, derived from Chen's context-free grammar,  for the  generating functions  of $G_n(x_1,x_2,x_3,x_4)$ and the refined Motzkin polynomials $M_n(u_1,u_2,u_3;v_1,v_2)$ defined in the introduction. Using binary trees, we can prove a unified formula that generalizes their two formulas.

We consider the following new refined Narayana polynomials.
\begin{definition}[New refined Narayana polynomials]
Let
$$
N_n=N_n(u_1,u_2,u_3;v_1,v_2):=\sum_{T\in\P_n}u_1^{\sleaf(T)} u_2^{\etleaf(T)} u_3^{\entleaf(T)} v_1^{\yerleaf(T)} v_2^{\syleaf(T)}
$$
for $n\geq2$ and $N_1=N_1(u_1,u_2,u_3;v_1,v_2):=u_1$ by convention. 
\end{definition}

The  polynomials $N_n$ for $2\leq n\leq5$ are listed as follows:
    \begin{align*}
       N_2=&u_1+u_2v_2,\\
        N_3=&u_1+u_1u_3+u_1v_2+u_2v_2+u_2v_1v_2,\\
        N_4=&u_2v_1^2v_2+u_1u_2v_2+u_1u_3v_1+u_1v_1v_2+u_2u_3v_2+u_2v_1v_2+u_2v_2^2\\
    &+u_1^2+2u_1u_3+2u_1v_2+u_2v_2+u_1,\\
        N_5=&u_2v_1^3v_2+2u_1u_2v_1v_2+u_1u_3v_1^2+u_1v_1^2v_2+u_2^2v_2^2+2u_2u_3v_1v_2+u_2v_1^2v_2\\
        &+2u_2v_1v_2^2+u_1^2u_3+u_1^2v_1+u_1^2v_2+4u_1u_2v_2+u_1u_3^2+2u_1u_3v_1+2u_1u_3v_2\\
        &+2u_1v_1v_5+u_1v_2^2+2u_2u_3v_2+u_2v_1v_2+2u_2v_2+3u_1^2+3u_1u_3+3u_1v_2\\
        &+u_2v_2+u_1.
    \end{align*}
    By Lemma~\ref{obv:rho2}, we have
$$N_n=\sum_{B\in\B_n}u_1^{\sleaf(B)} u_2^{\etleaf(B)} u_3^{\entleaf(B)} v_1^{\yerleaf(B)} v_2^{\syleaf(B)},$$
where $\sleaf(B)$ (resp.,~$\etleaf(B)$, $\entleaf(B)$, $\yerleaf(B)$, $\syleaf(B)$) is the number of right leaves (resp., left leaves whose parent has no right child, left leaves whose parent has right child, nodes  with only left child and whose left child has left child, nodes  with only left child and whose left child has no left child) of $B$.
Fig.~\ref{fig:B3} provides all binary trees in $\B_3$, where the statistics are marked by their corresponding variables.

\begin{figure}
    \centering
    \begin{tikzpicture}[scale=0.8]
        \draw[-] (0,4) -- (-1,3);
        \draw[-] (-2,2) -- (-1,3);

        \node at (0,4) [circle,fill,inner sep=1.5pt]{};\node at (0.1,3.7) {$v_1$};
        \node at (-1,3) [circle,fill,inner sep=1.5pt]{};\node at (-0.8,2.7) {$v_2$};
        \node at (-2,2) [circle,fill,inner sep=1.5pt]{};\node at (-2,1.7) {$u_2$};

        \draw[-] (3,4) -- (2,3);
        \draw[-] (3,2) -- (2,3);

        \node at (3,4) [circle,fill,inner sep=1.5pt]{};\node at (3.1,3.7) {$v_2$};
        \node at (2,3) [circle,fill,inner sep=1.5pt]{};\node at (1.8,2.7) {};
        \node at (3,2) [circle,fill,inner sep=1.5pt]{};\node at (3.1,1.7) {$u_1$};

        \draw[-] (6,4) -- (5,3);
        \draw[-] (6,4) -- (7,3);

        \node at (6,4) [circle,fill,inner sep=1.5pt]{};\node at (6,3.7) {};
        \node at (5,3) [circle,fill,inner sep=1.5pt]{};\node at (5,2.7) {$u_3$};
        \node at (7,3) [circle,fill,inner sep=1.5pt]{};\node at (7,2.7) {$u_1$};

        \draw[-] (9,4) -- (10,3);
        \draw[-] (10,3) -- (9,2);

        \node at (9,4) [circle,fill,inner sep=1.5pt]{};\node at (6,3.7) {};
        \node at (10,3) [circle,fill,inner sep=1.5pt]{};\node at (10.2,2.7) {$v_2$};
        \node at (9,2) [circle,fill,inner sep=1.5pt]{};\node at (9,1.7) {$u_2$};

        \draw[-] (12,4) -- (13,3);
        \draw[-] (13,3) -- (14,2);

        \node at (12,4) [circle,fill,inner sep=1.5pt]{};\node at (6,3.7) {};
        \node at (13,3) [circle,fill,inner sep=1.5pt]{};
        \node at (14,2) [circle,fill,inner sep=1.5pt]{};\node at (14,1.7) {$u_1$};
    \end{tikzpicture}
    \caption{3-node binary trees}
    \label{fig:B3}
\end{figure}
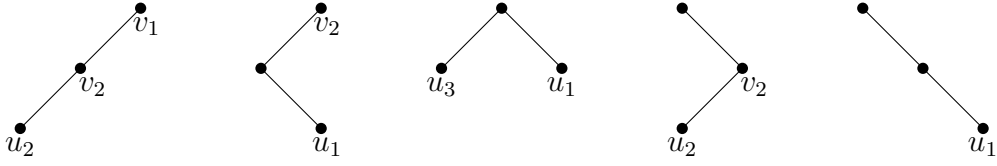

\begin{theorem}\label{thm:genc}
The expression of $N(t):=\sum_{n\geq1}N_nt^n$ is given by
\begin{equation}\label{gen Bt}
    N(t)=\frac{at^2-ct+1}{2t}-\frac{\sqrt{a^2t^4+2(bc-2u_1-2u_2v_2)t^3+(c^2-2b)t^2-2ct+1}}{2t},
\end{equation}
where $a=u_1-u_3+v_1-v_2$, $b=u_1+u_3-v_1+v_2$ and $c=v_1+1$.  
\end{theorem}
\begin{proof}
    We derive the functional equation for $N(t)$ using the classical  decomposition of binary trees $B=(B_g,r,B_d)$, where $r$ is the root of $B$ and $B_g,B_d$ are respectively the left and right branches  at $r$. We distinguish four cases below. 
    \begin{itemize}
\item [(1)] $B_g=B_d=\emptyset$. 
This case contributes $u_1t$ to $N(t)$.

\item [(2)] $B_g=\emptyset$ and $B_d\neq\emptyset$. 
This case contributes  $tN(t)$.
    
\item [(3)] $B_g\neq\emptyset$ and $B_d\neq\emptyset$. This case contributes $t(N(t)+u_3t-u_1t)N(t)$, as when $B_g$ contains only one node, then such node should contributes $u_3t$ other than $u_1t$.

\item [(4)] $B_g\neq\emptyset$ and $B_d=\emptyset$. We further distinguish three subcases.
\begin{itemize}
\item[(4a)] $B_g$ contains only one node. This case contributes $u_2v_2t^2$.
\item[(4b)] The root of $B_g$ has only right child. This case contributes $u_2t^2N(t)$.
\item[(4c)]  $B_g$ is not in cases (4a) and (4b). 
This case contributes $v_1t(N(t)-u_1t-tN(t))$.
  \end{itemize}
\end{itemize}
Summing over all the contributions from the above cases yields the functional equation for $N=N(t)$:
$$
N=t(u_1+N+(N+u_3t-u_1t)N+u_2v_2t+u_2tN+v_1(N-u_1t-tN)).
$$
Solving this equation gives~\eqref{gen Bt}.
\end{proof}

\begin{remark}
By~\eqref{eq:all}, we have 
$$
\yint(T)=n+1-2\oleaf(T)-\yleaf(T)
$$
for $T\in\P_n$. 
Thus, $N_n(u_1,u_2,u_3;v_1,v_2)$ are common generalization of $G_n(x_1,x_2,x_3,x_4)$ and $M_n(u_1,u_2,u_3;v_1,v_2)$. The main results in~\cite[Theorems~1.6 and~1.9]{Dong1} can then be deduced from Theorem~\ref{thm:genc} after appropriate substitutions. 
\end{remark}

\subsection{Deutsch's bijection via binary trees}
For $T\in\T_M$, the {\em level} of a node in $T$ is the number of edges on the path from the root to it. The {\em degree} of a node is the number of its children. 
Six concerned statistics on $T$  are: 
\begin{itemize}
\item The number of nodes of degree $q$ in $T$, denoted by $\deg_q(T)$.  In particular, $\deg_0(T)=\leaf(T)$, the number of leaves of $T$. 
\item The number of nodes of degree $q$ in odd levels of $T$, denoted by $\od_q(T)$.
\item The number of nodes in even levels of $T$, denoted by $\el(T)$.
\item The number of internal  nodes in even levels of $T$, denoted by $\elint(T)$.
\item The number of leaves in even levels of $T$, denoted by $\elleaf(T)$.
\item A non-root node  of $T$ is considered as {\em youngest} if it is the rightmost child of its parent. The number of youngest leaves of $T$ is denoted by $\ystleaf(T)$.
\end{itemize}
Inspired by the work of Kuznetsov, Pak and  Postnikov~\cite{KPP}, Deutsch~\cite{Deut} introduced a recursive bijection $\widehat{(\,)}: \P_n\rightarrow\P_n$  such that for each $T\in\mathcal{P}_n$:
 $$\deg_q(T)=\left\{\begin{array}{lll}\od_{q-1}(\widehat{T})&\text{if}&q\geq 1,\\
\el(\widehat{T})&\text{if}&q=0.\end{array}\right.$$
This bijection was extended independently by Lin--Ma~\cite{LM} and Chen--Fu~\cite{CFu} from plane trees to weakly increasing trees. For the sake of convenience, we recall the recursive construction of $\widehat{(\,)}: \T_M\rightarrow\T_M$ from~\cite{LM}.

 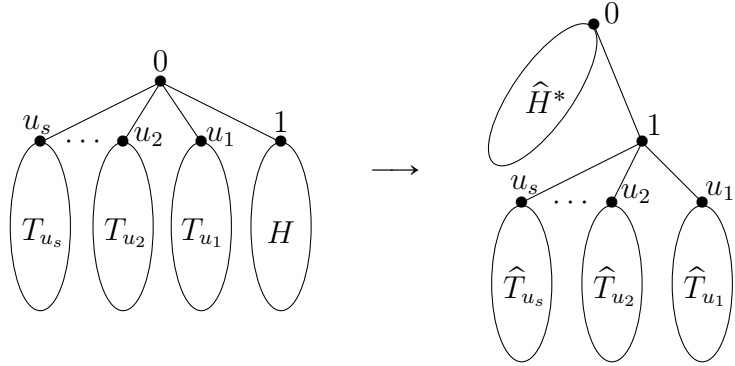
\begin{figure}
\centering
\begin{tikzpicture}[scale=0.4]
\node at (5,10) {$\bullet$};
\node at (5,10.7) {$0$};
\draw[-] (5,10) to (1,8);\draw[-] (5,10) to (3.75,8);\draw[-] (5,10) to (6.35,8);\draw[-] (5,10) to (9,8);
\node at (1,8) {$\bullet$};
\node at (0.9,8.6) {$u_s$};
\draw (1,5.2) ellipse (1 and 2.8);

\node at (2.5,8) {$\cdots$};

\node at (3.75,8) {$\bullet$};
\node at (4.6,8.2) {$u_2$};
\draw (3.75,5.2) ellipse (1 and 2.8);

\node at (6.35,8) {$\bullet$};
\node at (7,8.2) {$u_1$};
\draw (6.35,5.2) ellipse (1 and 2.8);
\node at (9,8) {$\bullet$};
\node at (9,8.7) {$1$};
\draw (9,5.2) ellipse (1 and 2.8);

\node at (1.1,5) {$T_{u_s}$};
\node at (3.85,5) {$T_{u_2}$};
\node at (6.4,5) {$T_{u_1}$};
\node at (9,5) {$H$};

\draw [rotate=-35] (9,18) ellipse (1 and 2.8) ;

\node at (12.8,7) {$\longrightarrow$};
\node at (19.4,11.9) {$\bullet$};
\node at (20,12.3) {$0$};

\node at (17.8,9.5) {$\widehat{H}^*$};
\draw[-] (19.4,11.9) to (21,8);
\node at (21,8) {$\bullet$};
\node at (21.4,8.6) {$1$};
\draw[-] (21,8) to (17,6);
\node at (17,6) {$\bullet$};
\node at (17.1,6.6) {$u_s$};
\draw (17,3.3) ellipse (1 and 2.6);
\draw[-] (21,8) to (23,6);
\draw (23,3.3) ellipse (1 and 2.6);
\node at (23,6) {$\bullet$};
\node at (23.6,6.4) {$u_1$};
\draw[-] (21,8) to (20,6);
\draw (20,3.3) ellipse (1 and 2.6);
\node at (20,6) {$\bullet$};
\node at (20.8,6.3) {$u_2$};
\node at (18.7,6) {$\cdots$};
\node at (17.1,3.2) {$\widehat{T}_{u_s}$};
\node at (20.1,3.2) {$\widehat{T}_{u_2}$};
\node at (23.1,3.2) {$\widehat{T}_{u_1}$};

\end{tikzpicture}
\caption{The construction of $\widehat{T}$.\label{deu:bij}}
\end{figure}

\begin{framed}
 \begin{center}
{\bf The map $\widehat{(\,)}: \T_M\rightarrow\T_M$} 
\end{center}

Firstly, set $\widehat{\emptyset}=\emptyset$.
Let $T\in\mathcal{T}_M$.  For a node $v$ of $T$ let $T_v$ denote the  subtree of $T$ rooted at $v$. Clearly, the rightmost child of the root $0$ in $T$ must be a node with label $1$ and let $H$ denote the subtree of $T$ rooted at this $1$. 
Suppose that  all children of the root $0$  other than the rightmost child in $T$ from right to left are $u_1,\ldots,u_s$. We construct recursively the weakly increasing tree $\widehat{T}$ as follows:
\begin{itemize}
\item Change the label $1$ of the root of $\widehat{H}$ to $0$ and denote by $\widehat{H}^*$ the resulting tree;
\item Attach an isolate node $1$ to the root $0$ of $\widehat{H}^*$ as its rightmost child;
\item Make $\widehat{T}_{u_1},\ldots, \widehat{T}_{u_s}$ the branches of the above new node $1$ from right to left.
\end{itemize}
 See Fig.~\ref{deu:bij} for the visualization of $\widehat{(\,)}$.
 \end{framed}

 The bijection $\widehat{(\,)}: \T_M\rightarrow\T_M$ can be constructed directly using binary trees as follows. 
Let $B\in\B_M$ be a binary tree. A node is called a {\em head} in $B$ if  either it is the root or it is a right child in $B$. Let $H(B)$ be the set of all heads in $B$. For instance, if $B$ is the second tree in Fig.~\ref{theta}, then $H(B)=\{1,2,4,8,9\}$.  Define the map $\theta:\B_M\rightarrow\B_M$ by 
$$
\theta(B)=\prod_{x\in H(B)}\phi_x(B).
$$
See Fig.~\ref{theta} for an example of $\theta$. 
 
 \begin{figure}
\centering

\begin{tikzpicture}[scale=0.9]

\node at (0,4) {$\bullet$}; \node at (0,4.3) {$0$};
\node at (-1,3) {$\bullet$}; \node at (-1.2,3.2) {$7$};
\node at (0,3) {$\bullet$};\node at (0,2.7) {$6$};
\node at (1,3) {$\bullet$};\node at (1.2,3.2) {$1$};
\node at (-1.5,2) {$\bullet$};\node at (-1.5,1.7) {$10$};
\node at (-0.5,2) {$\bullet$};\node at (-0.7,1.7) {$8$};
\node at (0.5,2) {$\bullet$};\node at (0.3,1.8) {$3$};
\node at (1.5,2) {$\bullet$};\node at (1.5,1.7) {$2$};
\node at (-0.5,1) {$\bullet$};\node at (-0.5,0.7) {$9$};
\node at (0.75,1) {$\bullet$};\node at (0.75,0.7) {$4$};
\node at (0.25,1) {$\bullet$};\node at (0.25,0.7) {$5$};

\draw[-] (0,4) -- (-1,3);
\draw[-] (0,4) -- (0,3);
\draw[-] (0,4) -- (1,3);
\draw[-] (-1,3) -- (-1.5,2);
\draw[-] (-1,3) -- (-0.5,2);
\draw[-] (1,3) -- (0.5,2);
\draw[-] (1,3) -- (1.5,2);
\draw[-] (-0.5,2) -- (-0.5,1);
\draw[-] (0.5,2) -- (0.75,1);
\draw[-] (0.5,2) -- (0.25,1);

\draw[->] (2,2) -- (2.6,2);
\node at (2.3,2.2) {$\rho$};

\draw[-] (4.8,3.6) -- (5.6,2.8);
\draw[-] (4.8,3.6) -- (3.8,2.8);
\draw[-] (3.8,2.8) -- (3,2);
\draw[-] (3,2) -- (3.8,1.2);
\draw[-] (3.8,1.2) -- (4.4,0.4);
\draw[-] (3.8,1.2) -- (3,0.4);
\draw[-] (5.6,2.8) -- (4.8,2);
\draw[-] (4.8,2) -- (5.6,1.2);
\draw[-] (5.6,1.2) -- (5,0.4);

\node at (4.8,3.6) {\magenta{$\bullet$}};\node at (4.8,3.9) {$1$};
\node at (3.8,2.8) {$\bullet$};\node at (3.8,3.1) {$6$};

\node at (5.6,2.8) {\magenta{$\bullet$}};\node at (5.8,2.8) {$2$};
\node at (3,2) {$\bullet$};\node at (3,1.7) {$7$};
\node at (4.8,2) {$\bullet$};\node at (4.6,2) {$3$};

\node at (3,0.4) {$\bullet$};\node at (3,0.1) {$10$};

\node at (3.8,1.2) {\magenta{$\bullet$}};\node at (4,1.4) {$8$};

\node at (5.6,1.2) {\magenta{$\bullet$}};\node at (5.8,1.2) {$4$};
\node at (4.4,0.4) {\magenta{$\bullet$}};\node at (4.4,0.1) {$9$};

\node at (5,0.4) {$\bullet$};\node at (5,0.1) {$5$};

\draw[->] (6,2) -- (6.6,2);
\node at (6.3,2.25) {$\theta$};

\draw[-] (8.4,3.6) -- (7.2,2.8);
\draw[-] (8.4,3.6) -- (9.6,2.8);
\draw[-] (8.8,2) -- (9.6,2.8);
\draw[-] (8.8,2) -- (9.6,1.2);
\draw[-] (8.8,0.4) -- (9.6,1.2);
\draw[-] (7.2,2.8) -- (8.4,0.4);
\draw[-] (10.4,0.4) -- (9.6,1.2);
\node at (7.2,2.8) {\magenta{$\bullet$}};\node at (7.2,3.1) {$2$};

\node at (8.4,3.6) {\magenta{$\bullet$}};\node at (8.4,3.9) {$1$};
\node at (9.6,2.8) {$\bullet$};\node at (9.6,2.5) {$6$};

\node at (8.8,2) {$\bullet$};\node at (8.8,2.4) {$7$};
\node at (9.6,1.2) {\magenta{$\bullet$}};\node at (9.6,1.5) {$8$};
\node at (8.8,0.4) {\magenta{$\bullet$}};\node at (9,0.3) {$9$};
\node at (10.4,0.4) {$\bullet$};\node at (10.1,0.3) {$10$};

\node at (8.4,0.4) {$\bullet$};\node at (8.1,0.4) {$5$};
\node at (8,1.2) {\magenta{$\bullet$}};\node at (7.8,1.2) {$4$};
\node at (7.6,2) {$\bullet$};\node at (7.4,2) {$3$};

\draw[->] (10.2,2) -- (10.9,2);
\node at (10.6,2.3) {$\rho^{-1}$};

\node at (12.6,4) {$\bullet$}; \node at (12.6,4.3) {$0$};
\draw[-] (12.6,4) -- (11.5,3.2);
\draw[-] (12.6,4) -- (13.7,3.2);
\node at (13.7,3.2) {$\bullet$};\node at (13.7,3.5) {$1$};
\draw[-] (13.7,3.2) -- (12.3,1.2);\node at (13,2.2) {$\bullet$};\node at (13,2.5) {$7$};
\node at (12.3,1.2) {$\bullet$};\node at (12.5,1.2) {$9$};
\draw[-] (13.7,3.2) -- (14.4,2.2);\node at (14.4,2.2) {$\bullet$};\node at (14.4,1.9) {$6$};
\draw[-] (13,2.2) -- (13.7,1.2);\node at (13.7,1.2) {$\bullet$};\node at (13.9,1.2) {$8$};
\draw[-] (13.7,1.2) -- (13.7,0.2);\node at (13.7,0.2) {$\bullet$};\node at (14,0.2) {$10$};

\draw[-] (11.5,3.2) -- (11.5,0.2);
\node at (11.5,3.2) {$\bullet$};\node at (11.5,3.5) {$2$};
\node at (11.5,2.2) {$\bullet$};\node at (11.7,2.2) {$3$};
\node at (11.5,1.2) {$\bullet$};\node at (11.7,1.2) {$4$};
\node at (11.5,0.2) {$\bullet$};\node at (11.7,0.2) {$5$};

\end{tikzpicture}
\caption{An example of the bijection $\Theta=\rho^{-1}\circ\theta\circ\rho$.\label{theta}}
\end{figure}
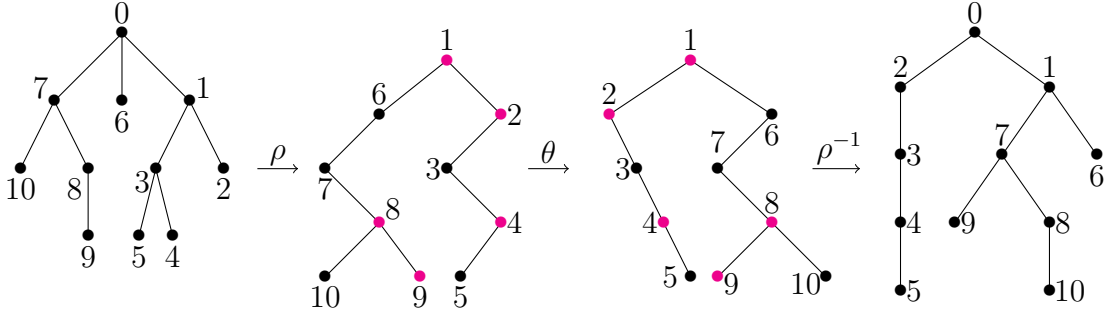
 
 \begin{theorem}
 Set $\Theta=\rho^{-1}\circ\theta\circ\rho$. For any $T\in\T_M$, $\Theta(T)=\widehat{T}$. 
 \end{theorem}
 \begin{proof}
 This can be proved by induction on the number of nodes of $T$, which is straightforward and will be left to the interested reader. 
 \end{proof}
 
 For a tree $B\in\B_M$ and a node $v$ of $B$, if there are $i-1$ right edges lying on the path from the root to $v$, then $v$ is said to be at {\em right-level} $i$. Let $\Orl(B)$ be the set of nodes at odd right-levels in $B$. For example, $\Orl(B)=\{1,2,4,8,9\}$ if $B$ is the third tree in Fig.~\ref{theta}. 
 
 \begin{lemma}\label{deu:lem}
 For any $B\in\B_M$, $H(B)=\Orl(\theta(B))$. 
 \end{lemma}
\begin{proof}
We aim to  prove that for any node $v$, 
\begin{equation*}
(\star)\qquad v\text{ is in $H(B)$}\Leftrightarrow v\text{ is in $\Orl(\theta(B))$}
\end{equation*}
 by induction on the {\em depth} of  $v$, that is the number of edges on the path from the root to $v$. 
 
 Clearly, $(\star)$ is true when $v$ is the root, the only node with depth $0$. Suppose that $v\in H(B)$ is a non-root node and let  $y$ be its parent in $B$. If $y\in H(B)$ (resp.,~$y\notin H(B)$), then $v$ is the left (resp.,~right) child of $y$ in $\theta(B)$. As $y$ has depth strictly less than that of $v$,  then by induction $y\in\Orl(\theta(B))$ (resp.~$y\notin\Orl(\theta(B))$), which implies $v\in\Orl(\theta(B))$. 
 The fact that $v\notin H(B)$ implies $v\notin\Orl(\theta(B))$ follows similarly. This proves $(\star)$ for all nodes $v$ by induction. 
\end{proof}
By Lemma~\ref{deu:lem}, the inverse of $\theta$ can be defined as 
$$
\theta^{-1}(B)=\prod_{x\in\Orl(B)}\phi_x(B)
$$
for $B\in\B_M$. Thus, the inverse of $\Theta$ is $\Theta^{-1}=\rho^{-1}\circ\theta^{-1}\circ\rho$. This direct construction of Deutsch's bijection $\Theta=\widehat{(\,)}$  enables us to observe its further features. 

Let $T\in\T_M$ and let $v$ be a node of $T$. We associate $v$ with a node $\tilde v$ as follows:
\begin{itemize}
\item[(i)] If $v$ is an internal node, then $\tilde v$ is the youngest child of $v$;
\item[(ii)] If $v$ is a leaf that is not youngest, then set $\tilde v=v$;
\item[(iii)] If $v$ is a youngest leaf, then $\tilde v$ is the first node that is not youngest in the path from $v$ to the root.  
\end{itemize}
For the first tree in Fig.~\ref{theta}, we have 
$$
\tilde 0=1, \tilde1=2, \tilde2=0, \tilde3=4, \tilde4=3, \tilde5=5, \tilde6=6, \tilde7=8, \tilde8=9, \tilde9=7, \widetilde{10}=10. 
$$
It is plain to see that $v\mapsto\tilde v$ is a bijection on all nodes of $T$.  

\begin{theorem}\label{thm:deu}
For $T\in\T_M$ and $v$ a node of $T$, then the following assertions hold.
\begin{itemize}
\item[(i)] If $v$ is an internal node in $T$ with degree $q$, then $\tilde v$ is an odd-level node with degree $q-1$ in $\widehat{T}$.
\item[(ii)]  If $v$ is a leaf that is not youngest in $T$, then $\tilde v$ is a leaf in even level in  $\widehat{T}$. 
\item[(iii)] If $v$ is a youngest leaf  and there are $k$ edges ($k>0$)  in the path from $v$ to $\tilde v$ in $T$, then $\tilde v$ is an internal node with degree $k$ in even-level in $\widehat{T}$. In particular, $\ystleaf(T)=\elint(\widehat T)$. 
\end{itemize}
Moreover, $\oleaf(T)=\oleaf(\widehat T)$. 
\end{theorem}
\begin{proof}
(i) Since $\tilde v$ is the youngest child, it becomes a head in $\rho(T)$, which is the leader of its {\em left chain} (i.e., the maximal path starting with $\tilde v$ compositing only left edges) whose members are its $q-1$ siblings. By Lemma~\ref{deu:lem}, $\tilde v$ turns to be a node at odd right-level under $\theta$. Thus, it becomes an odd-level node with degree $q-1$ in $\widehat T$. 

(ii) In this case $v=\tilde v$, which is a leaf that is not youngest in $T$. Thus, under $\rho$ it becomes a non-head without right child. By Lemma~\ref{deu:lem}, $\tilde v$ turns to be a node without right child at even right-level under $\theta$. Eventually, it turns to be a leaf at even-level in $\widehat{T}$ under $\rho^{-1}$. 

(iii) Since $\tilde v$ is the first node that is not youngest in the path from $v$ to the root in $T$, if letting $P$ be the path from $\tilde v$ to $v$, then all the nodes in $P$ are youngest, except $\tilde v$. Then under $\rho$, $P$ becomes the {\em right chain} (i.e., the maximal path starting with $\tilde v$ compositing only right edges) with $\tilde v$ as the leader. This right chain turns to a maximal {\em hoe} with $\tilde v$ as the leader, where a hoe is a path
begins with a right edge  and continues with only left edges. Eventually, this hoe becomes a {\em claw} under $\rho^{-1}$, i.e., the node $\tilde v$ with all the remaining $k$ nodes in the hoe as children in $\widehat T$. Thus, $\tilde v$ is of degree $k$ and as $\tilde v$ is not youngest in $T$, $\tilde v$ lies at even level in $\widehat T$ according to Lemma~\ref{deu:lem}. 

Finally, by Lemma~\ref{lem:phi}, old leaves in $T$ become leaves in $\rho(T)$, which are preserved under $\theta$. Thus, $\oleaf(T)=\oleaf(\widehat T)$, as desired. 
\end{proof}

One interesting consequence of the plane tree case of Theorem~\ref{thm:deu}  is the following  symmetry. 
\begin{corollary}
For $n\geq1$, the pair $(\oleaf,\elint)$  is symmetric on $\P_n$.  In particular, the two statistics `$\oleaf$' and `$\elint$' are equidistributed on $\P_n$ and so 
$$
|\{T\in\P_n: \elint(T)=k\}|=2^{n-2k+1}{n-1\choose 2k-2}C_{k-1}. 
$$
\end{corollary}
\begin{proof}
By Theorem~\ref{thm:deu}, the pair $(\oleaf,\elint)$ has the same joint distribution as the pair $(\oleaf,\ystleaf)$ on $\P_n$. By the minor symmetry of plane trees, we see that $(\oleaf,\ystleaf)$ is symmetric  on $\P_n$, from which the symmetry $(\oleaf,\elint)$  follows.

The second statement then follows from the fact 
\begin{equation}\label{eq:cde}
|\{T\in\P_n: \oleaf(T)=k\}|=2^{n-2k+1}{n-1\choose 2k-2}C_{k-1},
\end{equation}
which was proved in~\cite[Proposition~2]{Chen1}. Alternatively, formula~\eqref{eq:cde} can also be deduced from~\eqref{rel:motz} by setting $x_2=y_1=y_2=1$.
\end{proof}

\begin{remark}
By the minor symmetry of plane trees, the two triples $(\oleaf,\ystleaf,\elint)$ and $(\ystleaf,\oleaf,\elint)$ are equidistributed on $\P_n$. In particular, $(\ystleaf,\elint)$ is symmetric  on $\P_n$. By Theorem~\ref{thm:deu}, the two statistics ``$\ystleaf$'' and ``$\elint$'' are equidistributed on $\T_M$. However, $(\ystleaf,\elint)$ is not symmetric on $\T_n$, the set of all increasing trees on $[n]$. 

Note also that ``$\oleaf$'' and ``$\elint$'' are not equidistributed on $\T_n$. 
\end{remark}

Another consequence of  Theorem~\ref{thm:deu} for plane trees and the minor symmetry of plane trees is the following new interpretation for the refined Narayana polynomials. 

\begin{corollary}\label{cor:elintleaf}
 For $n\geq1$, we have the equidistribution
$$
\sum_{T\in\P_n} x^{\oleaf(T)}y^{\yleaf(T)}=\sum_{T\in\P_n} x^{\elint(T)}y^{\elleaf(T)}. 
$$
\end{corollary}

\section{Applications to permutation patterns and statistics}
\label{sec:3}

In this section, we use a classical bijection  between  increasing binary trees $\I_n:=\B_{[n]}$ and permutations $\S_n$ of $[n]$ to transform statistics from trees to permutations.

Given a word $w=w_1w_2\cdots w_n$ of positive integers with no repeated letters, we can write it as $w=\sigma i\tau$ with $i$ the least element of $w$.  Set $\lambda(\emptyset)=\emptyset$ and define $\lambda(w)$ recursively as $(\lambda(\sigma), i,\lambda(\tau))$, the binary tree  rooted at $i$ with left branch $\lambda(\sigma)$ and right branch $\lambda(\tau)$. See Fig.~\ref{bij:lambda} for an example of $\lambda$. The map $\lambda: \S_n \rightarrow \I_n$ is a bijection (see~\cite[Page~23]{St0}). The inverse map $\lambda^{-1}$ can be viewed as projecting the nodes of the increasing binary trees downwards.

\begin{figure}
\centering
\begin{tikzpicture}[scale=0.27]

\node at  (0,2){$\longrightarrow$};\node at  (0,3){$\lambda$};
\node at  (-10,2){$2~ 5~ 6~ 4~ 3~ 8~ 7~ 1~ 9~ 10$};
\draw[-](11,9) to (8,6);
\draw[-](11,9) to (17,3);\draw[-](8,6) to (14,0);
\draw[-](11,-3) to (14,0);\draw[-](11,3) to (5,-3);\draw[-](8,-6) to (5,-3);
\node at  (11,9){$\bullet$};\node at  (11.6,9.4){$1$};
\node at  (8,6){$\bullet$};\node at  (7.3,6.3){$2$};
\node at  (14,6){$\bullet$};\node at  (14.7,6.5){$9$};
\node at  (17,3){$\bullet$};\node at  (17.9,3.6){$10$};
\node at  (11,3){$\bullet$};\node at  (11.7,3.5){$3$};
\node at  (14,0){$\bullet$};\node at  (14.6,0.5){$7$};
\node at  (11,-3){$\bullet$};\node at  (11,-2){$8$};
\node at  (8,0){$\bullet$};\node at  (7.3,0.3){$4$};
\node at  (5,-3){$\bullet$};\node at  (4.3,-2.7){$5$};
\node at  (8,-6){$\bullet$};\node at  (7.2,-6){$6$};
\end{tikzpicture}
\caption{An example of the bijection $\lambda$.\label{bij:lambda}}
\end{figure}
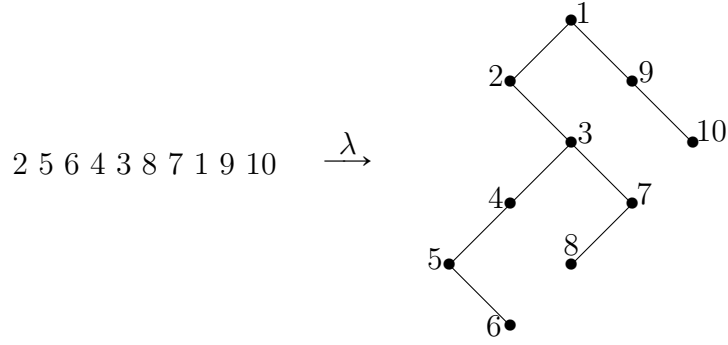

\subsection{Refinements of the peak statistic on permutations}

Given a permutation $\pi=\pi_1\pi_2\dots\pi_n\in\S_n$, a letter $\pi_i$  ($1\leq i\leq n$) is called a {\em peak} (resp.,~{\em double descent, double ascent})  of $\pi$ if  $\pi_{i-1}<\pi_{i}>\pi_{i+1}$ (resp.,~$\pi_{i-1}>\pi_{i}>\pi_{i+1}$, $\pi_{i-1}<\pi_{i}<\pi_{i+1}$), where by convention $\pi_0=\pi_{n+1}=0$. 
We further refer to a peak $\pi_i$ as of {\em type 1} if $\pi_{i-1}\leq\pi_{i+1}$ and of {\em type 2} otherwise. 
Let $\pk(\pi)$ (resp.,~$\pk_1(\pi)$, $\pk_2(\pi)$, $\dd(\pi)$, $\da(\pi)$) be the number of peaks (resp., type 1 peaks, type 2 peaks, double descents, double ascents) of $\pi$.

\begin{lemma}\label{lem:2types}
Let $\pi\in\S_n$ and $B=\lambda(\pi)$. For $n\geq2$ and $i\in[n]$, we have
\begin{itemize}
\item[(i)] $\pi_i$ is a peak of type $1$ in $\pi$ iff $\pi_i$ is a left leaf in $B$;
\item[(ii)] $\pi_i$ is a peak of type $2$ in $\pi$ iff $\pi_i$ is a right leaf in $B$;
\item[(iii)] $\pi_i$ is a double descent in $\pi$ iff $\pi_i$ has only left child in $B$;
\item[(iv)] $\pi_i$ is a double ascent in $\pi$ iff $\pi_i$ has only right child in $B$. 
\end{itemize}
\end{lemma}
\begin{proof}
Properties (iii) and (iv) are known in~\cite[Page~24]{St0}, as well as the fact that $\pi_i$ is a peak in $\pi$   iff $\pi_i$ is a leaf in $B$. Assertions (i) and (ii) then follow from the observation that $\pi_i$ is a left (resp.,~right) leaf in $B$ implies $\pi_{i-1}<\pi_{i+1}$ (resp.,~$\pi_{i-1}>\pi_{i+1}$).
\end{proof}

It is well known (see~\cite[Page~22]{St0}) that the {\em Eulerian polynomial} $A_n(x)$ enumerates permutations in $\S_n$ by number of descents. 
Introduce the refinement of Eulerian polynomials for $n\geq1$ by  
$$A_n(x,y,z_1,z_2)=\sum_{\pi\in\S_n}x^{\dd(\pi)}y^{\da(\pi)}z_1^{\pk_1(\pi)}z_2^{\pk_2(\pi)}.$$
The first three values of $A_n(x,y,z_1,z_2)$ are
    \begin{align*}
        A_1(x,y,z_1,z_2)&=z_1,\\
        A_2(x,y,z_1,z_2)&=xz_1+yz_2,\\
        A_3(x,y,z_1,z_2)&=x(x+y)z_1+y(x+y)z_2+2z_1z_2.
    \end{align*}
    Comparing Lemma~\ref{lem:2types} with Lemma~\ref{lem:phi}, we have  
    $$
    A_n(x,y,z_1,z_2)=\sum_{T\in\T_n}x^{\yleaf(T)}y^{\yint(T)}z_1^{\eleaf(T)}z_2^{\sleaf(T)}
    $$
    for $n\geq2$. 
    We are interested in computing the generating function 
    $$
    A(x,y,z_1,z_2;t)=\sum_{n\geq1}A_n(x,y,z_1,z_2)\frac{t^n}{n!}
    $$
    for two reasons. First, Carlitz and Scoville~\cite{CS74} proved the nice generating function formula 
    \begin{equation}\label{eq:CS74}
    A(x,y,z,z;t)=\frac{uv(e^{ut}-e^{vt})}{ue^{vt}-ve^{ut}},
    \end{equation}
    where $u+v=x+y$ and $uv=z$. Second, Dong et al.~\cite{Dong1} has obtained explicit algebraic generating function (see Theorem~\ref{thm:genc} for a refinement) for  plane trees by the quadruple of statistics $(\eleaf,\sleaf,\yleaf,\yint)$. Unfortunately, we are unable to find any explicit formula for $A(x,y,z_1,z_2;t)$. Instead, we can prove that  $A(x,y,z_1,z_2;t)$ satisfies a Riccati equation. 
   A  grammatical calculus approach to Carlitz-Scoville's formula~\eqref{eq:CS74} was given by Fu~\cite{Fu} and it would be interesting to see whether her approach could be adopted to the refinement $A(x,y,z_1,z_2;t)$.

\begin{theorem}
    The generating function $A=A(x,y,z_1,z_2;t)$ satisfies the Riccati equation:
    \begin{equation}\label{genpeak}
        \frac{\partial A}{\partial t}=A^2+((z_2-z_1)t+x+y)A+(z_2-z_1)yt+z_1.   
    \end{equation}    
\end{theorem}
\begin{proof}
 For $1\leq i\leq n$, let $A_{n,i}=A_{n,i}(x,y,z_1,z_2):=\sum_{\pi\in\S_{n,i}}x^{\dd(\pi)}y^{\da(\pi)}z_1^{\pk_1(\pi)}z_2^{\pk_2(\pi)}$, where   $\S_{n,i}$ is the set of permutations $\pi\in\S_n$ with $\pi_i=1$. For convenience, we write $A_n=A_{n}(x,y,z_1,z_2)$. 
    For $n\geq4$ and $\pi\in\S_{n,i}$, we write $\pi=\sigma1\tau$ and consider the following five cases.
    \begin{enumerate}[(1)]
        \item If $i=1$, then we have
        $$(\dd,\pk_1,\pk_2)(\pi)=(\dd,\pk_1,\pk_2)(\tau)\text{ and } \da(\pi)=1+\da(\tau).$$
        Thus, $A_{n,1}=yA_{n-1}$.
        \item If $i=n$, then we have
        $$(\da,\pk_1,\pk_2)(\pi)=(\da,\pk_1,\pk_2)(\sigma)\text{ and } \dd(\pi)=1+\dd(\sigma).$$
        Thus,  $A_{n,n}=xA_{n-1}$.
        \item If $i=2$, then we have
       $$(\dd,\da,\pk_2)(\pi)=(\dd,\da,\pk_2)(\tau)\text{ and } \pk_1(\pi)=1+\pk_1(\tau).$$
        Thus, $A_{n,2}=(n-1)z_1A_{n-2}$.
         \item If $i=n-2$, then we have
       $$(\dd,\da,\pk_1)(\pi)=(\dd,\da,\pk_1)(\sigma)\text{ and } \pk_2(\pi)=1+\pk_2(\sigma).$$
        Thus, $A_{n,2}=(n-1)z_2A_{n-2}$.
        \item If  $3\leq i \leq n-2$, then we have
        $$(\dd,\da,\pk_1,\pk_2)(\pi)=(\dd,\da,\pk_1,\pk_2)(\sigma)+(\dd,\da,\pk_1,\pk_2)(\tau).$$
        Thus, $A_{n,i}=\binom{n-1}{i-1}A_{i-1}A_{n-i}$ for $3\leq i \leq n-2$.
    \end{enumerate}
        Summing over all cases gives 
        \begin{equation*}
                A_{n}=(x+y)A_{n-1}+(n-1)(z_1+z_2)A_{n-2}+\sum_{i=2}^{n-3}\binom{n-1}{i}A_{i}A_{n-1-i}
         \end{equation*}
         for $n\geq4$. 
       Involving the first three values of $A_n(x,y,z_1,z_2)$, we can turn this recurrence relation into the Riccati equation~\eqref{genpeak}.
\end{proof}

We could not solve the Riccati equation~\eqref{genpeak} to get explicit expression for $A(x,y,z_1,z_2;t)$.
However, the special case that $x=y=z_2=1$ can be solved.

\begin{proposition}
    Let $A(z;t)=\sum_{n\geq1}\sum_{\pi\in\S_n}z^{\pk_1(\pi)}\frac{t^n}{n!}$. Then, 
    \begin{equation}\label{genlpeak0}
        A(z;t)=\frac{e^{(z-1)t^2/2}}{1-\int_0^{t}e^{(z-1)x^2/2}dx}+(z-1)t-1.
    \end{equation}
\end{proposition}

\begin{proof}
    By setting $x=y=z_2=1$ and $z_1=z$ in~\ref{genpeak}, we get the Riccati equation for $A(z;t)$
    \begin{equation}\label{genlpeak}
        \frac{\partial A(z;t)}{\partial t}=A(z;t)^2+((1-z)t+2)A(z;t)+(1-z)t+z.
    \end{equation}
    We observe that $(z-1)t-1$ is a particular solution of~\eqref{genlpeak}.
    Setting $B(z;t)=A(z;t)-(z-1)t+1$,  we get the Bernoulli differential equation:
    \begin{equation*}
        \frac{\partial B(z;t)}{\partial t}+(1-z)tB(z;t)=B(z;t)^2.
    \end{equation*}
   Letting $C(z;t)=B(z;t)^{-1}$, we obtain the following linear first-order ordinary differential equation:
    $$
        \frac{\partial C(z;t)}{\partial t}+(z-1)tC(z;t)+1=0.
    $$
     By solving this differential equation, we can express $C(z;t)$ as:
    $$
    C(z;t)=c_1e^{(1-z)t^2/2}-c_2e^{(1-z)t^2/2}\int e^{(z-1)t^2/2}dt,
    $$
    where $c_1$ is an arbitrary constant and $c_2$ is an arbitrary positive constant. Therefore, we can write:
    $$
    A(z;t)=\biggl(c_1e^{(1-z)t^2/2}-c_2e^{(1-z)t^2/2}\int e^{(z-1)t^2/2}dt\biggr)^{-1}+(z-1)t-1.
    $$
   Involving $A(z;0)=0$ and $A'(z;0)=z$, we get~\eqref{genlpeak0}.
\end{proof}

For $\sigma\in\S_k$ and a word $w=w_1w_2\cdots w_n$ of positive integers, an occurrence of  (consecutive pattern) $\sigma$ in $w$ is a subword $w_iw_{i+1}\cdots w_{i+k-1}$ that is order isomorphism to $\sigma$. Let $\underline\sigma(w)$ denote the number of occurrences of   $\sigma$ in $w$. Then, $\pk_1(\pi)=\underline{132}(0\pi)$ and $\pk_2(\pi)=\underline{231}(\pi0)$ for any $\pi\in\S_n$ with $n\geq2$. Consider the enumerator of permutations by the number of occurrences of $132$ as
$$
B_n(z):=\sum_{\pi\in\S_n} z^{\underline{132}(\pi)}.
$$
Elizalde and Noy~\cite[Theorem~4.1]{EN} proved that 
$$
\sum_{n\geq0}B_n(z)\frac{t^n}{n!}=\frac{1}{1-\int_0^{t}e^{(z-1)x^2/2}dx},
$$
where by convention $B_0(z)=1$. 
Comparing with~\eqref{genlpeak0} we get the following relationship. 
\begin{corollary}
For $n\geq2$, we have 
$$
\sum_{\pi\in\S_n}z^{\pk_1(\pi)}=\sum_{k=0}^{\lfloor n/2\rfloor}\frac{(z-1)^k}{2^k}{n\choose k}B_{n-2k}(z). 
$$
\end{corollary}
 It would be interesting to find a combinatorial proof of the above relationship. 

\subsection{Symmetries in $312$-avoiding permutations
} 

A permutation $\pi\in \S_n$ is said to be {\em$312$-avoiding} if  there are no  indices $1\leq i<j<k\leq n$ such that $\pi_{i}>\pi_k>\pi_j$. If we denote $\S_n(312)$ the set of all $312$-avoiding permutations in $\S_n$, then a classical  result of Knuth (see~\cite[Section~2.1.1]{PPW}) asserts  that $|\S_n(312)|=C_n$.  In this subsection, we prove two symmetries in $312$-avoiding permutations involving maximum number of non-overlapping ascents/descents and some consecutive patterns of length $4$.

Let us recall some classical permutation statistics.  An {\em ascent} (resp., A~{\em descent}) in a permutation $\pi\in\S_n$ is an index $i\in[n-1]$ such that $\pi_i<\pi_{i+1}$ (resp.,~$\pi_i>\pi_{i+1}$). Denote by $\ASC(\pi)$ and $\DES(\pi)$ the set of all ascents and the set of all descents of $\pi$, respectively. Let $\asc(\pi):=|\ASC(\pi)|$ and $\des(\pi):=|\DES(\pi)|$.
Furthermore, the number of  {\em left-to-right} (resp.,~{\em right-to-left}) {\em minima} of $\pi$ is denoted by $\lmin(\pi)$ (resp.,~$\rmin(\pi)$). The {\em left descending run} (resp.,~{\em right ascending run}) of $\pi$, namely the largest $j\in[n]$ such that the subsequence $\pi_1\pi_2\cdots\pi_j$ (resp.,~$\pi_{n-j+1}\pi_{n-j+2}\dots\pi_n$) of $\pi$ is decreasing (resp.,~increasing), is denoted as $\ldr(\pi)$ (resp.,~$\rar(\pi)$). 

Two  indices  $i$ and $j$, $i<j$,  are said to be {\em overlapping} if $j = i+1$. Kitaev~\cite{Kit} introduced  the {\em maximum number of non-overlapping ascents} (resp.,~{\em descents}) in a permutation $\pi$, denoted $\mna(\pi)$ (resp.,~$\mnd(\pi)$). For example,  $\mna(231485697)= 3$ and $\mnd(126543) = 2$. 
Recently, Kitaev and Zhang~\cite{KZ} investigated the statistics ``$\mna$'' and ``$\mnd$'' over stack-sortable permutations (i.e., $231$-avoiding permutations) and proved a symmetry that is equivalent to the following result. 

\begin{theorem}[Kitaev-Zhang~\cite{KZ}]\label{thm:KZ}
The following two $8$-tuples of statistics
  \begin{align*}
       &(\mnd, \mna,  \des,\asc, \ldr, \rar,   \lmin, \rmin)\\
      &(\mna, \mnd, \asc,\des, \rar, \ldr,  \rmin, \lmin).   
    \end{align*}
     are equidistributed over $\S_n(312)$ for $n\geq1$.
\end{theorem}

As an application of the minor symmetry of binary trees, we can generalize the above result as follows. 

\begin{theorem}\label{sym:KZ}
For $\pi\in\S_n$, let $\widetilde\ASC(\pi):=\{n-i: i\in\ASC(\pi)\}$. There exists an involution $\Lambda$ on $\S_n(312)$ that transforms the tripe $(\DES,\pk_1,\lmin)$ to $(\widetilde\ASC,\pk_2,\rmin)$. In particular, $\Lambda$ provides an involution proof of Theorem~\ref{thm:KZ}. 
\end{theorem}

\begin{proof}
For $\pi\in\S_n(312)$, it is plain to see that the nodes of the increasing binary tree $\lambda(\pi)$ are exactly labeled according to their preorder (see Fig.~\ref{bij:lambda} where the permutation is $312$-avoiding). Thus, we can remove all labels of $\lambda(\pi)$ and denote the resulting binary tree by $\tilde\lambda(\pi)$.  It is clear that $\tilde\lambda:\S_n(312)\rightarrow\B_n$ is a bijection. Now define $\Lambda=\tilde\lambda^{-1}\circ\phi\circ\tilde\lambda$, where $\phi$ is the minor symmetry on $\B_n$.
See Fig.~\ref{fig:Lambda} for an example of $\Lambda$.

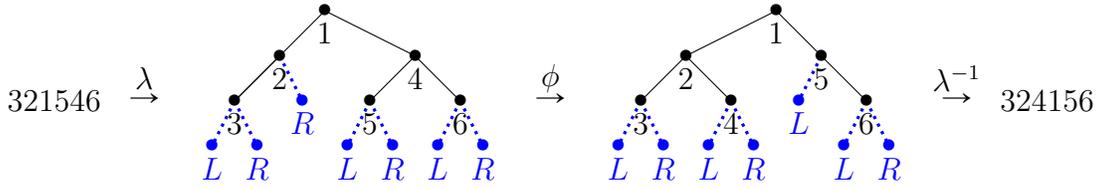
\begin{figure}
\begin{center}
\begin{tikzpicture}[scale=0.3]

\draw[very thick,dotted,blue](6,2) to (7,0);
\draw[very thick,dotted,blue](8,4) to (9,2);
\draw[very thick,dotted,blue](12,2) to (11,0);
\draw[very thick,dotted,blue](12,2) to (13,0);
\draw[very thick,dotted,blue](16,2) to (15,0);
\draw[very thick,dotted,blue](16,2) to (17,0);
\draw[very thick,dotted,blue](6,2) to (5,0);

\node at  (-2,2){$321546$};\node at  (2,2){$\rightarrow$};
\node at  (2,3){$\lambda$};
\node at  (10,6){$\bullet$};\node at  (10,5){$1$};
\node at  (8,4){$\bullet$};\node at  (8,3){$2$};
\node at  (6,2){$\bullet$};\node at  (6,1){$3$};
\node at  (14,4){$\bullet$};\node at  (14,3){$4$};
\node at  (12,2){$\bullet$};\node at  (12,1){$5$};
\node at  (16,2){$\bullet$};\node at  (16,1){$6$};
\draw[-](10,6) to (8,4);\draw[-](10,6) to (14,4);
\draw[-](14,4) to (12,2);\draw[-](14,4) to (16,2);
\draw[-](8,4) to (6,2);
\node at  (5,0){\blue{$\bullet$}};\node at  (5,-1){{\blue{$L$}}};
\node at  (7,0){\blue{$\bullet$}};\node at  (7,-1){{\blue{$R$}}};
\node at  (9,2){\blue{$\bullet$}};\node at  (9,1){{\blue{$R$}}};
\node at  (11,0){\blue{$\bullet$}};\node at  (11,-1){{\blue{$L$}}};
\node at  (13,0){\blue{$\bullet$}};\node at  (13,-1){{\blue{$R$}}};
\node at  (15,0){\blue{$\bullet$}};\node at  (15,-1){{\blue{$L$}}};
\node at  (17,0){\blue{$\bullet$}};\node at  (17,-1){{\blue{$R$}}};

\node at (20,2){$\rightarrow$};\node at (20,3){$\phi$};

\draw[very thick,dotted,blue](24,2) to (23,0);
\draw[very thick,dotted,blue](24,2) to (25,0);
\draw[very thick,dotted,blue](28,2) to (27,0);
\draw[very thick,dotted,blue](28,2) to (29,0);
\draw[very thick,dotted,blue](32,4) to (31,2);
\draw[very thick,dotted,blue](34,2) to (33,0);
\draw[very thick,dotted,blue](34,2) to (35,0);

\node at  (30,6){$\bullet$};\node at  (30,5){$1$};
\node at  (28,2){$\bullet$};\node at  (28,1){$4$};
\node at  (24,2){$\bullet$};\node at  (24,1){$3$};
\node at  (32,4){$\bullet$};\node at  (32,3){$5$};
\node at  (26,4){$\bullet$};\node at  (26,3){$2$};
\node at  (34,2){$\bullet$};\node at  (34,1){$6$};
\draw[-](30,6) to (26,4);\draw[-](30,6) to (32,4);
\draw[-](26,4) to (24,2);\draw[-](26,4) to (28,2);
\draw[-](32,4) to (34,2);\draw[-](8,4) to (6,2);
\node at  (23,0){\blue{$\bullet$}};\node at  (23,-1){{\blue{$L$}}};
\node at  (25,0){\blue{$\bullet$}};\node at  (25,-1){{\blue{$R$}}};
\node at  (27,0){\blue{$\bullet$}};\node at  (27,-1){{\blue{$L$}}};
\node at  (29,0){\blue{$\bullet$}};\node at  (29,-1){{\blue{$R$}}};
\node at  (31,2){\blue{$\bullet$}};\node at  (31,1){{\blue{$L$}}};
\node at  (33,0){\blue{$\bullet$}};\node at  (33,-1){{\blue{$L$}}};
\node at  (35,0){\blue{$\bullet$}};\node at  (35,-1){{\blue{$R$}}};

\node at (38,2){$\rightarrow$};\node at (38,3){$\lambda^{-1}$};
\node at  (42,2){$324156$};
\end{tikzpicture}
\end{center}
\caption{An example of the involution $\Lambda$.\label{fig:Lambda}}
\end{figure} 

Let $B=\lambda(\pi)$. To facilitate the proof, we introduce  the {\em complement of a binary tree}. We assign  leaves labeled by $L$ or $R$ to $B$ to obtain a full binary tree, ensuring that all original nodes in $B$ become internal nodes. We use $\bar{B}$ to denote such a complement of $B$ where the added leaf in $\bar B$ is labeled by $L$ iff it is a left leaf. See Fig.~\ref{fig:complement} for an example of $\bar B$. 
It is plain to see that the index $i$ is a descent of $\pi$ iff the $(i+1)$-th (under preorder)  new leaf in $\bar B$ is $R$. Thus, $\DES(\pi)=\widetilde\ASC(\Lambda(\pi))$. 
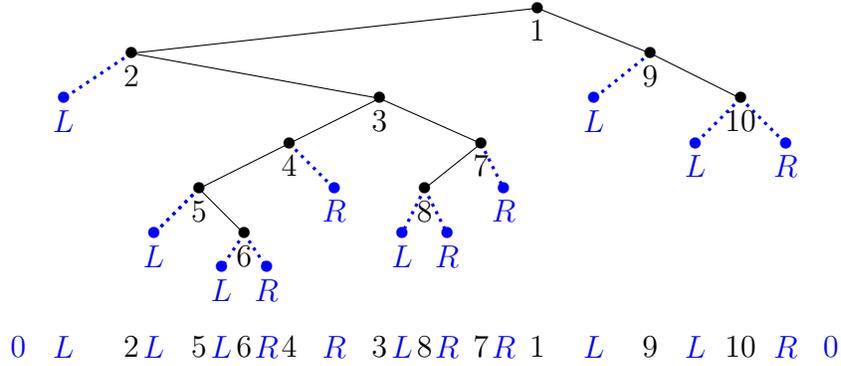
\begin{figure}
\centering
\begin{tikzpicture}[scale=0.3]

\draw[very thick,dotted,blue](-3,8) to (-6,6);
\draw[very thick,dotted,blue](0,2) to (-2,0);\draw[very thick,dotted,blue](2,0) to (1,-1.5);
\draw[very thick,dotted,blue](2,0) to (3,-1.5);\draw[very thick,dotted,blue](4,4) to (6,2);
\draw[very thick,dotted,blue](10,2) to (11,0);\draw[very thick,dotted,blue](10,2) to (9,0);
\draw[very thick,dotted,blue](0,2) to (-2,0);\draw[very thick,dotted,blue](12.5,4) to (13.5,2);
\draw[very thick,dotted,blue](20,8) to (17.5,6);\draw[very thick,dotted,blue](24,6) to (22,4);
\draw[very thick,dotted,blue](24,6) to (26,4);

\node at  (-6,-5){\blue{$L$}};\node at  (-2,-5){\blue{$L$}};
\node at  (1,-5){\blue{$L$}};\node at  (3,-5){\blue{$R$}};
\node at  (6,-5){\blue{$R$}};\node at  (9,-5){\blue{$L$}};
\node at  (11,-5){\blue{$R$}};\node at  (13.5,-5){\blue{$R$}};
\node at  (17.5,-5){\blue{$L$}};\node at  (22,-5){\blue{$L$}};
\node at  (26,-5){\blue{$R$}};
\node at  (-8,-5){\blue{$0$}};
\node at  (28,-5){\blue{$0$}};

\node at  (15,10){$\bullet$};\node at  (15,9){$1$};
\node at  (-3,8){$\bullet$};\node at  (-3,7){$2$};
\node at  (20,8){$\bullet$};\node at  (20,7){$9$};
\node at  (8,6){$\bullet$};\node at  (8,5){$3$};
\node at  (24,6){$\bullet$};\node at  (24,5){$10$};
\node at  (4,4){$\bullet$};\node at  (4,3){$4$};
\node at  (12.5,4){$\bullet$};\node at  (12.5,3){$7$};
\node at  (0,2){$\bullet$};\node at  (0,1){$5$};
\node at  (10,2){$\bullet$};\node at  (10,1){$8$};
\node at  (2,0){$\bullet$};\node at  (2,-1){$6$};

\draw[-](15,10) to (-3,8);\draw[-](15,10) to (20,8);
\draw[-](-3,8) to (8,6);\draw[-](20,8) to (24,6);
\draw[-](8,6) to (4,4);\draw[-](8,6) to (12.5,4);
\draw[-](4,4) to (0,2);\draw[-](12.5,4) to (10,2);
\draw[-](0,2) to (2,0);

\node at  (15,-5){$1$};
\node at  (-3,-5){$2$};
\node at  (20,-5){$9$};
\node at  (8,-5){$3$};
\node at  (24,-5){$10$};
\node at  (4,-5){$4$};
\node at  (12.5,-5){$7$};
\node at  (0,-5){$5$};
\node at  (10,-5){$8$};
\node at  (2,-5){$6$};

\node at  (-6,6){\blue{$\bullet$}};\node at  (-6,5){\blue{$L$}};
\node at  (-2,0){\blue{$\bullet$}};\node at  (-2,-1){\blue{$L$}};
\node at  (1,-1.5){\blue{$\bullet$}};\node at  (1,-2.5){\blue{$L$}};
\node at  (3,-1.5){\blue{$\bullet$}};\node at  (3,-2.5){\blue{$R$}};
\node at  (6,2){\blue{$\bullet$}};\node at  (6,1){\blue{$R$}};
\node at  (9,0){\blue{$\bullet$}};\node at  (9,-1){\blue{$L$}};
\node at  (11,0){\blue{$\bullet$}};\node at  (11,-1){\blue{$R$}};
\node at  (13.5,2){\blue{$\bullet$}};\node at  (13.5,1){\blue{$R$}};
\node at  (17.5,6){\blue{$\bullet$}};\node at  (17.5,5){\blue{$L$}};
\node at  (22,4){\blue{$\bullet$}};\node at  (22,3){\blue{$L$}};
\node at  (26,4){\blue{$\bullet$}};\node at  (26,3){\blue{$R$}};

\end{tikzpicture}
\caption{An example of the complement of a binary tree $\bar B$.\label{fig:complement}}
\end{figure}

The fact that $\pk_1(\pi)=\pk_2(\Lambda(\pi))$ follows from Lemma~\ref{lem:2types} and the construction of $\Lambda$. 

The {\em left arm} (resp.,~{\em right arm}) of $B$ is the maximal chain starting from the root and compositing only left (resp.,~right) edges. Form the inverse map $\lambda^{-1}$, it is clear that $\lmin(\pi)$ (resp.,~$\rmin(\pi)$) equals the number of nodes in the left (resp.,~right) arm of $B$. The fact that $\lmin(\pi)=\rmin(\Lambda(\pi))$ then follows from the construction of $\Lambda$. 

Finally, as $\DES(\pi)$ (or $\widetilde\ASC(\pi)$) determines the sextuple  $(\mnd, \mna,  \des,\asc, \ldr, \rar)(\pi)$, we see that $\Lambda$ provides an involution proof of Theorem~\ref{thm:KZ}. 
\end{proof}

For a permutation $\pi$, let $\st_1(\pi):=\underline{1324}(0\pi)$ and $\st_2(\pi):=\underline{3241}(\pi0)$.

   \begin{theorem}
     There exists an involution $\Upsilon$ on $\S_n(312)$ that preserves the triple $(\pk,\dd,\da)$ but exchanges the pair $(\st_1,\st_2)$.
 \end{theorem}
 \begin{proof}
 
Recall the bijection $\tilde\lambda:\S_n(312)\rightarrow\B_n$ in the proof of Theorem~\ref{sym:KZ} and the involution $\psi$ on $\B_n$ introduced in the proof of Theorem~\ref{Thm:sym2}. Set $\Upsilon=\tilde\lambda^{-1}\circ\psi\circ\tilde\lambda$. 
It follows from Lemma~\ref{lem:2types} and the construction of $\psi$ that $\Upsilon$ preserves the triple $(\pk,\dd,\da)$.

     To demonstrate that $\Upsilon$ exchanges the pair $(\st_1,\st_2)$, we claim that a  consecutive $1324$ (resp.,~$3241$) pattern in $\pi\in\S_n(312)$ is maps to a left (resp.,~right) leaf whose parent has right (resp.,~left) child in the binary tree $B=\lambda(\pi)$. If $\pi$ contains a consecutive $1324$ pattern $\pi_{i-1}\pi_i\pi_{i+1}\pi_{i+2}$, then $\pi_i$ is a peak of type 1 and so $\pi_i$ is a left leaf by Lemma~\ref{lem:2types}~(i). Moreover,  $\pi_{i+2}>\pi_{i+1}$ implies that $\pi_{i+1}$ has right child. 
Conversely, if the node $\pi_i$ is a left  leaf in $B$ whose parent has right  child, then $\pi_{i-1}\pi_i\pi_{i+1}$ is a peak of type 1 by Lemma~\ref{lem:2types}~(i). Furthermore, as $\pi_{i+1}$ is the parent of $\pi_i$ in $B$, it has right child. Thus, $\pi_{i-1}<\pi_{i+1}<\pi_i<\pi_{i+2}$ and so $\pi_{i-1}\pi_i\pi_{i+1}\pi_{i+2}$ forms a consecutive $1324$ pattern in $\pi$. This proves that $\st_1(\pi)=\entleaf(B)$. The fact that $\st_2(\pi)=\suleaf(B)$ follows by similar discussions. As $\psi$ exchanges the pair $(\entleaf,\suleaf)$ on $\B_n$, $\Upsilon$ exchanges the pair $(\st_1,\st_2)$ on 
$\S_n(312)$. This completes the proof of the theorem. 
\end{proof}

 \section*{Acknowledgement} 
 This work was supported by the National Science Foundation of China (grants  12271301 \& 12322115) and the Fundamental Research Funds for the Central Universities.

\end{document}